\documentclass[12pt,hypertexnames=false]{amsart}

\usepackage{amsmath}
\usepackage{amsfonts,mathtools}
\usepackage{amssymb,color}
\usepackage{amsthm,doi}
\usepackage{thmtools}
\usepackage{tikz,color}
\usepackage{pgfplots}
\usepackage{graphicx}
\usepackage{enumitem}
\pgfplotsset{compat=1.15}
\usepackage[square, numbers]{natbib}
\usepackage{comment}
\usepackage{mathrsfs}
\usepackage{etoolbox}

\hypersetup{
	colorlinks   = true, 
	urlcolor     = blue, 
	linkcolor    = blue, 
	citecolor   = red 
}

\newcommand{\N}{\mathbb{N}}
\newcommand{\Z}{\mathbb{Z}}

\newcommand{\R}{\mathbb{R}}

\newcommand{\Rg}{\R^d_\gamma}
\newcommand{\Sd}{\mathbb{S}^{d-1}}
\newcommand{\D}{\mathcal{D}}

\newcommand{\V}{\mathcal{V}}
\newcommand{\Ro}{\mathcal{R}}
\newcommand{\Lo}{\mathcal{L}}

\newcommand{\Bo}{\mathcal{B}}
\newcommand{\To}{\mathcal{T}}
\newcommand{\torus}{\mathbb{T}^d}

\newcommand{\A}{\mathcal{A}}

\newcommand{\Hyp}{H_{\operatorname{hyp}}^1}
\newcommand{\HypO}{H_{\operatorname{hyp},0}^1}
\newcommand{\Hyptr}{H_{\operatorname{hyp},\tr}^1}
\newcommand{\Hyptrp}{H_{\operatorname{hyp},\tr_+}^1}
\newcommand{\Hyptrm}{H_{\operatorname{hyp},\tr_-}^1}
\newcommand{\Hyptrpm}{H_{\operatorname{hyp},\tr_\pm}^1}
\newcommand{\ktr}{K\textsubscript{tr}}
\newcommand{\kre}{K\textsubscript{refl}}
\newcommand{\kto}{K\textsubscript{per}}
\newcommand{\Crefl}{C^1_{c, \operatorname{refl}}(\overline{\D})}

\newcommand{\HOv}{H^1_{0,v}}
\newcommand{\Hyprefl}
{H_{\operatorname{hyp},\operatorname{refl}}^1}

\DeclareMathOperator{\supp}{supp}
\DeclareMathOperator{\dist}{dist}
\DeclareMathOperator{\tr}{tr}
\DeclareMathOperator{\Tr}{Tr}

\DeclareMathOperator{\loc}{loc}

\def\hmath$#1${\texorpdfstring{{\rmfamily\textit{#1}}}{#1}}

\newtheorem{lemma}{Lemma}[section]
\newtheorem{theorem}[lemma]{Theorem}
\newtheorem{corollary}[lemma]{Corollary}
\newtheorem{proposition}[lemma]{Proposition}

\theoremstyle{definition}
\newtheorem{definition}[lemma]{Definition}

\newtheorem{remark}[lemma]{Remark}
\newtheorem{problemx}{Problem}
\newenvironment{problem}[1]
 {\renewcommand\theproblemx{#1}\problemx}
 {\endproblemx}

\let\oldtocsection=\tocsection
\let\oldtocsubsection=\tocsubsection
\let\oldtocsubsubsection=\tocsubsubsection
\renewcommand{\tocsection}[2]{\hspace{0em}\oldtocsection{#1}{#2}}
\renewcommand{\tocsubsection}[2]{\hspace{1em}\oldtocsubsection{#1}{#2}}
\renewcommand{\tocsubsubsection}[2]{\hspace{2em}\oldtocsubsubsection{#1}{#2}}

\usepackage{hyperref}
\PassOptionsToPackage{plainpages=false, pdfpagelabels,hypertexnames = false}{hyperref}

\numberwithin{equation}{section}

\author[L. Valentini]{Lisa Valentini}
\address[L. V.]{ 
Department of Pure Mathematics and Mathematical Statistics
\\ University of Cambridge \\
Wilberforce Road, Cambridge CB3 0WA, United Kingdom}
\email{lv390@cam.ac.uk}

\title[Well-posedness of the Kolmogorov-Fokker-Planck equation]{Well-posedness of the Kolmogorov-Fokker-Planck equation on bounded domains}



\begin{document}

\begin{abstract}
We establish well-posedness of the stationary Kolmogorov equation on a bounded domain, with {either} spherical {or Gaussian velocities}, subject to either inflow boundary conditions or specular reflection. For the sake of completeness, we also include the problem on the torus, which has already been solved in \cite{Albritton2024VariationalEquation}. {Although the natural trace estimate with weight $|n_x\cdot v|$ does not hold in general in either setting, we provide a natural functional framework in which the}
trace is defined via the transport operator. We prove a Poincaré-type inequality with trace, which is an essential step towards the well-posedness of the inflow problem without damping. Moreover, we obtain a one-sided energy estimate for the trace with weight $|n_x\cdot v|$, in which the outflow (resp. inflow) flux is bounded by the energy inside the domain and the inflow (resp. outflow) flux. Finally, for the bounded-velocity model, we complement the existing well-posedness theory for the stationary Kolmogorov equation with inflow conditions on the hypoelliptic boundary and Dirichlet conditions on the velocity boundary.
\end{abstract}

\maketitle
\tableofcontents

\section{Introduction}

\subsection{The Kolmogorov equation}

For $d\geq 2$, let $\Rg$ be the Euclidean space $\R^d$ endowed with the standard Gaussian measure
\begin{equation}
    d\gamma(v) = \frac{1}{(2\pi)^{d/2}}e^{-\frac{|v|^2}{2}} dv,
\end{equation}
and let $\Sd\subseteq\R^d$ be the unit sphere. We develop a well-posedness theory for the stationary Kolmogorov-Fokker-Planck equation on bounded spatial domains $\Omega$,
either with Gaussian-weighted velocities
\begin{equation}
    \label{eq:kolm_gaus}
    v \cdot \nabla_xf = {\Delta}_vf - v \cdot \nabla_v f+  S,
    \qquad 
    (x,v)\in \Omega \times\Rg,
\end{equation}
or spherical velocities
\begin{equation}
    \label{eq:kolm_rad}
    v \cdot \nabla_x f = {\Delta}_vf+  S,
    \qquad 
    (x,v)\in \Omega \times\Sd,
\end{equation}
with solution $f=f(x,v)$ and source $S=S(x,v)$. 

Kolmogorov equations arise in the theory of stochastic processes as linear second-order parabolic equations with a non-negative characteristic form \cite{Chandrasekhar1943StochasticAstronomy}. The prototypical time-dependent Kolmogorov equation on the Euclidean phase-space $\R^{2d}$, namely
\begin{equation}\label{eq:kolm_intro}
    (\partial_{t} + v \cdot \nabla_x) f = \Delta_v f, \quad (t, x,v) \in (0,+\infty) \times \R^{2d},
\end{equation}
was first introduced by Kolmogorov \cite{Kolmogorov1934ZufalligeBewegung} as a model for random motion in which velocity evolves diffusively (via Brownian noise) and position changes according to that velocity. The solution $f(t,x,v)$ is interpreted as the density evolution of the particles that, at time $t$, occupy position $x$ and velocity $v$ in the phase space. The equation combines two contributions: the drift term $\partial_{t} + v \cdot \nabla_x$, representing deterministic transport, and the diffusion term $\Delta_v$, which models diffusion in velocity. More generally, the equation \eqref{eq:kolm_intro} is called the kinetic Fokker-Planck equation (or Kolmogorov-Fokker-Planck) when pure diffusion $\Delta_v$ is replaced with $\nabla_v \cdot (\nabla_v + v)$, and it is called the Kramers-Fokker-Planck equation when confinement ${\bf b}\cdot \nabla_v$ due to a systematic force ${\bf b}={\bf b}(x)$ is included. 

The Gaussian weight arises naturally in the study of the kinetic Fokker–Planck equation. Indeed, the Maxwellian is a stationary solution of the equation, and expressing solutions relative to this equilibrium transforms the equation into one whose natural reference measure is the Gaussian measure. Its stationary version is \eqref{eq:kolm_gaus}. 

The equation \eqref{eq:kolm_rad} and its time-dependent version can be seen in the framework of the linear Boltzmann equation, which describes the statistical propagation of neutral particles through a background scattering medium. Depending on the physical application, this overarching framework is referred to either as the Radiative Transfer Equation when modelling photon transport \cite{Chandrasekhar1960RadiativeTransfer, Ishimaru1978WaveMedia}, with applications in astrophysics \cite{Mihalas1999FoundationsHydrodynamics}, atmospheric radiation \cite{Mishchenko2006MultipleBackscattering} and optical imaging \cite{Kim2003LightTissue}, or as the Neutron Transport Equation when modelling reactor core physics in nuclear engineering \cite{Weinberg1959TheReactors,Bell1970NuclearTheory,Duderstadt1976NuclearAnalysis, Dautray2000MathematicalTechnology}. For photons, because they travel at the constant speed of light, the velocity space naturally restricts to the unit directional sphere $v\in\Sd$. While neutron interactions fundamentally involve variable kinetic energies, a standard paradigm known as mono-energetic (or one-speed) transport theory assumes a uniform particle speed, thereby restricting to $v\in\Sd$. Under the physical assumption that scattering is highly forward-peaked (meaning particles experience only infinitesimal angular deflections when colliding), the classical integral scattering operator asymptotically reduces to the Laplace-Beltrami operator on $\Sd$, via the so-called grazing collision or small-angle limit \cite{POMRANING1992THELIMIT,Leakeas2001GeneralizedScattering}. The transport model is then reduced to a Kolmogorov-Fokker-Planck equation with spherical velocity, whose simplest example is \eqref{eq:kolm_rad}. Therefore, we will sometimes refer to equation \eqref{eq:kolm_rad} as the Kolmogorov equation for radiative transfer.

We will also consider the confined-velocity model 
\begin{equation}
    \label{eq:kolm_bounded}
    v \cdot \nabla_x f = {\Delta}_vf+  S,
    \qquad 
    (x,v)\in \Omega \times\V,
\end{equation}
where $\V \subseteq \R^d$ is a Lipschitz bounded domain. Contrary to equations \eqref{eq:kolm_gaus} and \eqref{eq:kolm_rad}, boundary value problems for \eqref{eq:kolm_bounded} involve boundary data on the velocity boundary $\Omega \times \partial\V$ as well.

The Ornstein–Uhlenbeck operator $\Delta_v - v \cdot \nabla_v$ on $\Rg$, the Laplace-Beltrami operator $\Delta_v$ on $\Sd$, and the Laplace operator $\Delta_v$ on $\V$ are the natural diffusive self-adjoint operators for the respective velocity models. The three equations \eqref{eq:kolm_gaus}, \eqref{eq:kolm_rad} and \eqref{eq:kolm_bounded} can be written as
\begin{equation}
    \label{eq:kolm_both}
    cf + v \cdot \nabla_x f = \A_vf+  S,
    \qquad 
    (x,v)\in \D=\Omega \times V,
\end{equation}
where $c\in\R$ and
\begin{equation*}
    \A_v =
    \begin{cases}
        \Delta_v - v \cdot \nabla_v & V=\Rg, \\
        {\Delta}_v            & V=\Sd, \V.
    \end{cases}
\end{equation*}
When $c > 0$, the shift $cf$ models attenuation (or damping) due to interactions with the background medium. From the analytical point of view, it also provides direct $L^2$-coercivity. The limiting case $c=0$, corresponding to the absence of damping, is more delicate, since this coercivity is lost.

Despite being degenerate, in the sense that the coefficient matrix of second-order derivatives is positive semidefinite, the Kolmogorov-Fokker-Planck equation is hypoelliptic, meaning that if the source is smooth, the solution is also smooth (according to the definition of H\"ormander \cite{Hormander1967HypoellipticEquations}). In \cite{Kolmogorov1934ZufalligeBewegung}, Kolmogorov gave the explicit expression for the fundamental solution of \eqref{eq:kolm_intro}, showing the hypoellipticity of the Kolmogorov operator explicitly. Later, H\"ormander \cite{Hormander1967HypoellipticEquations} introduced a family of second-order differential operators shaped after the Kolmogorov operator, and gave a geometric characterisation of their hypoellipticity. This theory applies to Kolmogorov-Fokker-Planck equations, and results regarding regularity of their solutions rely on subelliptic estimates -- see, for instance, \cite{Pascucci2004TheEquations, Golse2019HarnackEquation, Guerand2022QuantitativeTheory, Guerand2023Log-transformEquations, Anceschi2022AEquations}. However, when formulating boundary value problems on bounded domains, the degeneracy of the equations implies that the standard notion of trace available for elliptic problems no longer applies. 
In fact, the natural trace estimate for the relevant kinetic function space does not hold in general, as shown in \cite{Niebel2026SharpTheory}. Therefore, boundary value problems cannot be defined via trace operators as in the elliptic setting, but need a different functional framework.

\subsection{Content and structure of the paper}
Throughout the paper, we consider both Gaussian velocities ($V=\Rg$) and spherical velocities ($V=\Sd$). 
We prove well-posedness for the stationary Kolmogorov equation \eqref{eq:kolm_both} in three different scenarios: the case where $\Omega$ is a bounded domain with either inflow boundary condition (Problem~\ref{prob:trace_G}) or specular reflection (Problem~\ref{prob:refl}), and the case where $\Omega=\torus$ (Problem~\ref{prob:torus}{}). The latter was solved already in \cite{Albritton2024VariationalEquation} for Gaussian velocities, and we include it here for the sake of completeness. 

In particular, our main contributions are as follows.
\begin{itemize}
    \item  For $\Omega$ a bounded Lipschitz domain, we introduce the space $\Hyptr(\D)$ of functions $f\in\Hyp(\D)$ (see Definition \ref{def:Hyp}) for which there exists $f_\Gamma\in L^2(\partial\D,|n_x \cdot v|)$ such that the transport term $v \cdot \nabla_xf$ satisfies integration by parts against $C^1_c(\overline{\D})$ in the dual sense, with $f_\Gamma$ appearing in the boundary term (Definition \ref{def:hypA}). We call $f_\Gamma$ the trace of $f$.
    \item For $\Omega$ a bounded Lipschitz domain, we prove an a priori bound on the average of $f\in\Hyptr(\D)$ and, consequently, a Poincaré-type inequality (Theorem \ref{th:poincare1}).
    \item For $\Omega$ sufficiently regular, we prove well-posedness for the inflow problem Problem \ref{prob:trace_G} in $\Hyptr(\D)$ (Theorem~\ref{th:wellpos_inflow_initial}); in particular, the Poincaré-type inequality is crucial to recover the $L^2$ control in the equation without damping.
    \item For $\Omega$ sufficiently regular, we obtain a one-sided trace energy estimate with natural weight $|n_x\cdot v|$ for functions in $\Hyptr(\D)$: the outflow (resp. inflow) flux is bounded by the energy inside the domain and the inflow (resp. outflow) flux (Theorem \ref{th:strong_trace_0}).
    \item For $\Omega$ sufficiently regular, we show well-posedness for Problem \ref{prob:refl}{} (Theorem \ref{th:wellpos_refl_initial}).
\end{itemize}

For the confined-velocity model $V=\V$, the corresponding Dirichlet problem was studied by Avelin and Hou \cite{Avelin2026WeakDomains}. We complement their well-posedness result with energy estimates (Theorem \ref{thm:bounded-velocity-wellposedness}).

When $\Omega$ is a bounded domain, we assume its boundary is $C^{1,\alpha_V}$, where 
\begin{equation}\label{eq:alpha}
    \alpha_V \coloneqq 
    \begin{cases}
        1   &V=\Rg, \\
        1/3 &V=\Sd, \V.
    \end{cases}
\end{equation}
The natural $L^2(\partial\D,|n_x \cdot v|)$ trace estimate for the solution space $\Hyp(\D)$ of the Kolmogorov equation \eqref{eq:kolm_both} on $\D= \Omega \times V$ does not hold when $\Omega$ is $C^{1,1}$ and $V=\Rg$; it holds for $C^{1,1/2}$ domains when $V=\Sd,\V$, and fails for some domain with lower regularity \cite{Niebel2026SharpTheory}. In this work, we give well-posedness results despite the (possible) failure of the natural trace estimate. However, even though the definition of problems \ref{prob:trace} and \ref{prob:refl}, the functional framework provided by $\Hyptr(\D)$ and $\Hyprefl(\D)$, and the key Poincaré inequality (Theorem \ref{th:poincare1}) require merely Lipschitz regularity for $\Omega$, we were not able to prove well-posedness for $\Omega$ Lipschitz, and we leave it as an open question. In fact, we still rely on some sub-optimal trace operator to prove the energy estimates of the well-posedness results: the $C^{1,\alpha_V}$ regularity assumption ensures that the trace operator with weight $|n_x \cdot v|^2$  is bounded \cite{Niebel2026SharpTheory}.

In Section \ref{sec:results}, we define the three problems for the spherical and Gaussian velocity models; we state the main results and discuss them in relation to the previous literature. In Section \ref{sec:hypo_setup}, we analyse the hypoelliptic solution space $\Hyptr(\D)$ and address the Poincaré-type inequality. In Section \ref{sec:renorm}, we give a renormalisation formula for solutions and deduce comparison and maximum principles. In Section \ref{sec:well-posedness_proof}, we finally prove well-posedness and the one-sided trace estimate. In Section \ref{sec:bounded}, we prove energy estimates for the confined-velocity model.

\subsubsection{Notation}
We use $(\Omega, dx)$ to denote either a bounded Lipschitz domain in $\R^d$ or the torus $\torus=\R^d / \Z^d$, with Lebesgue measure. Throughout Sections \ref{sec:results}--\ref{sec:well-posedness_proof}, we denote by $(V,dv)$ either the unit sphere $(\Sd,d\theta)$ or the Gaussian space $(\Rg,d\gamma)$ with their respective measure. For $V=\Sd$, ${\Delta}_v$ and ${\nabla}_v$ denote the Laplace-Beltrami operator and the tangent gradient, respectively. For either $V$, we write
$$\D \coloneqq \Omega \times {V}$$ 
for the phase space, and we shorten $dz \coloneqq dx\,dv$ on $\D$. We adopt the following notation for the Bochner space
\begin{equation*}
    L_x^2H^1_v(\D) \coloneqq L^2(\Omega; H^1({V})),
    \qquad 
    \|f\|_{L_x^2H^1_v(\D)} = \|f\|_{L^2(\D)} + \|{\nabla}_vf\|_{L^2(\D)}.
\end{equation*}
Since ${V}$ is without boundary, we also have
\begin{gather*}
    L^2_xH^{-1}_v(\D) \coloneqq (L^2_xH^1_v(\D))',
    \\
    H^1_0(\D)= L^2(\Omega,H^1({V}))\cap L^2({V},H^1_0(\Omega)).
\end{gather*}
For $T\in L^2_xH^{-1}_v(\D)$ and $f\in L^2_xH^1_v(\D)$, we use $\langle T, f\rangle$ to denote the dual pairing
\begin{equation*}
    \langle T,f \rangle 
    \coloneqq \int_\Omega \langle T(x, \cdot), f(x, \cdot) \rangle_{H^{-1}(V), H^1(V)} \, dx.
\end{equation*}

For a function $h \in L^1(\D)$, we denote the averages on $\D$ and ${V}$ by
\begin{equation*}
    \langle h \rangle_\D \coloneqq \frac{1}{|\D|}\int_{\D} h(x,v) \ dz,
    \qquad \text{and} \qquad
    \langle h \rangle_{{V}}(x) \coloneqq \frac{1}{|V|}\int_{{V}} h(x,v) \ dv.
\end{equation*}
For $T\in L^2_xH^{-1}_v(\D)$ we write 
\begin{equation*}
    \langle T \rangle_\D \coloneqq \frac{1}{|\D|} \langle T,1\rangle.
\end{equation*}

For $\Omega$ a bounded Lipschitz domain, the phase-space boundary is $$\Gamma \coloneqq \partial\Omega \times {V},$$ and we shorten $d\sigma \coloneqq ds \, dv$, where $ds$ is the Hausdorff measure on $\partial\Omega$. We denote by $\Tr_x$ the standard $x$-trace operator for $H^1(\D)$. The boundary $\Gamma$ splits into the inflow part $\Gamma_-$, the outflow part $\Gamma_+$, and the grazing part $\Gamma_0$,
\begin{equation*}
    \Gamma_{\pm} \coloneqq \{(x,v)\in\Gamma \mid \pm n_x \cdot v > 0 \}
    \quad \text{and} \quad 
    \Gamma_0 \coloneqq \{(x,v)\in\Gamma \mid n_x \cdot v = 0 \},
\end{equation*}
where $n_x$ is the unit outward normal vector at $x$, which is well-defined for a.e. $x \in\partial\Omega$ -- see Figure \ref{fig:geometry}.
We will abbreviate $L^2(\Gamma, |n_x \cdot v|) \coloneqq L^2(\Gamma, |n_x \cdot v| \, d\sigma)$.

\begin{figure}[t]
\centering
\includegraphics[width=0.3\textwidth]{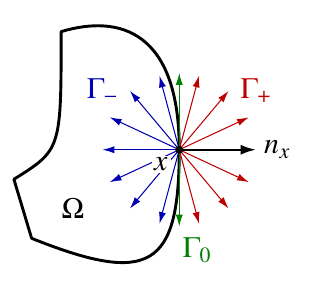}
\caption{Vectors $v\in {V}$ such that $(x,v)$ belongs to $\Gamma_+$, $\Gamma_-$ or $\Gamma_0$ respectively.}
\label{fig:geometry}
\end{figure}

Finally, we write $a \lesssim b$ if there exists a constant $C(\Omega,d)>0$ such that $ a \leq C(\Omega,d) b$, and subscripts will highlight any further dependence -- for instance  $a \lesssim_c b$ indicates that the implicit constant $C(\Omega,d,c)$ depends on $c$ as well.

\bigskip
\noindent {\bf Acknowledgements and AI tools disclosure.} 
We would like to thank Clément Mouhot for his valuable guidance and suggestions. The author is supported by a Gates Cambridge Scholarship issued by the Gates Cambridge Trust. 
Gemini was used to search for relevant literature on radiative and neutron transport. Outside of this AI tool use, the text
of this paper was human-generated.

\section{Results and background}
\label{sec:results}

We study the Kolmogorov equation \eqref{eq:kolm_both} with $S\in L_x^2H_v^{-1}(\D)$ and $c \in \R$. All results are valid both for $V=\Sd$ and for $V=\Rg$.

\subsection{The problems}
\label{sec:problems}
First, consider the case where $\Omega$ is a bounded domain. 

\begin{definition}[Solution in the interior of $\D$] \label{def:interior}
    Let $\Omega$ be a bounded Lipschitz domain, $S\in L_x^2H_v^{-1}(\D)$ and $c \in \R$. A \emph{weak solution to \eqref{eq:kolm_both} in the interior of $\D$} is a function $f \in L^2_xH^1_v(\D)$ that, for every $\varphi\in C^1_c({\D})$, satisfies
\begin{equation}
    \langle S,  \varphi \rangle
    = \int_{\D} \left[{\nabla}_vf \cdot {\nabla}_v\varphi - f(v\cdot \nabla_x \varphi) + c f \varphi \right] dz.
    \label{eq:kolm_rad_weak_interior}
\end{equation}
\end{definition}

\begin{definition}[Solution in $\D$ up to the boundary] \label{def:up_to_boundary}
    Let $\Omega$ be a bounded Lipschitz domain, $S\in L_x^2H_v^{-1}(\D)$ and $c \in \R$. A \emph{weak solution to \eqref{eq:kolm_both} in $\D$ up to the boundary} is a function $f\in L^2_xH^1_v(\D)$ for which there exists a (unique) $f_{\Gamma} \in L^2(\Gamma, |n_x \cdot v|)$, called \emph{trace} of $f$, such that, for every $\varphi\in C^1_c(\overline{\D})$, the pair $(f,f_\Gamma)$ satisfies
    \begin{equation}\label{eq:kolm_rad_weak}
        \langle S , \varphi \rangle
        = \int_{\D} \left[{\nabla}_vf \cdot {\nabla}_v\varphi - f(v\cdot \nabla_x \varphi) + c f \varphi\right] dz
        + \int_{\Gamma} (n_x \cdot v) f_{\Gamma} \varphi \ d\sigma.
    \end{equation}
\end{definition}

Let us introduce the boundary conditions.

\begin{problem}{\ktr($\Bo$)}[Boundary value problem on bounded $\Omega$, with trace]\label{prob:trace}
    Let $\Omega$ be a bounded Lipschitz domain, $S\in L_x^2H_v^{-1}(\D)$ and $c \in \R$. Consider a boundary operator $$\Bo : L^2(\Gamma, |n_x \cdot v|) \to L^2(\Gamma_-, |n_x \cdot v|).$$ A \emph{weak solution to \eqref{eq:kolm_both} in $\D$ with trace satisfying boundary condition $f=\Bo f$ on $\Gamma_-$} is a weak solution in $\D$ up to the boundary such that $f_{\Gamma}=\Bo f_{\Gamma}$ on $\Gamma_-$.
\end{problem}

We consider the following boundary operators. Let $h \in L^2(\Gamma, |n_x \cdot v|)$.
\begin{itemize}
    \item {\bf Inflow boundary operator:}
    \begin{equation*}
        \Bo_G h(x,v) \coloneqq G(x,v) \qquad \text{for } (x,v)\in\Gamma_-
        \csdef{@currentlabel}{\ktr($\Bo_G$)} \label{prob:trace_G}
    \end{equation*}
    for a boundary function $G \in L^{2}(\Gamma_-, |n_x\cdot v|)$. The case $G=0$ is called the absorbing boundary condition.
    \item {\bf Specular reflection operator:}
    \begin{equation*}
        \Bo_\Ro h(x,v) \coloneqq h(x, v-2(n_x \cdot v) n_x) \qquad \text{for } (x,v)\in\Gamma_-.
        \csdef{@currentlabel}{\ktr($\Bo_\Ro$)} \label{prob:trace_Ro}
    \end{equation*}
    \item {\bf Periodic boundary condition on $\Omega = (0,1)^d$:}
    \begin{equation*}
        \Bo_\To h(x,v) \coloneqq 
        \begin{cases}
            h(x+e_i,v)  &   \text{if} \ x_i=0 \\
            h(x-e_i,v)  &   \text{if} \ x_i=1
        \end{cases}
        \qquad \text{for } (x,v)\in\Gamma_-.
        \csdef{@currentlabel}{\ktr($\Bo_\To$)} \label{prob:trace_To}
\end{equation*}
\end{itemize}

\begin{remark}
    Whenever we consider $\Bo_G$, we are implicitly fixing a boundary function $G \in L^{2}(\Gamma_-, |n_x\cdot v|)$; when we consider $\Bo_\To$, we implicitly assume $\Omega = (0,1)^d$.
\end{remark} 

Define the following subspace of $C^1_c(\overline{\D})$ of functions satisfying specular reflection:
\begin{gather*}
    \Crefl \coloneqq \{\varphi\in C^1_c(\overline{\D}) \mid \varphi=\Bo_\Ro \varphi \ \text{on} \ \Gamma_-\}.
\end{gather*}

Let us now introduce two problems without boundary and therefore without trace.

\begin{problem}{\kre}[Specular reflection, without trace]
\label{prob:refl}
    Let $\Omega$ be a bounded Lipschitz domain, $S\in L_x^2H_v^{-1}(\D)$ and $c \in \R$. A \emph{weak solution to \eqref{eq:kolm_both} on $\D$ satisfying the specular reflection boundary condition} is a function $f\in L^2_xH^1_v(\D)$ that, for every $\varphi\in\Crefl$, satisfies
    \begin{equation}\label{eq:kolm_rad_weak_refl}
        \langle S , \varphi \rangle
        = \int_{\D} \left[{\nabla}_vf \cdot {\nabla}_v\varphi - f(v\cdot \nabla_x \varphi) + c f \varphi \right] dz.
    \end{equation}
\end{problem}

\begin{problem}{\kto}[Periodic on $\torus$, without trace]
\label{prob:torus}
    Let $\Omega = \torus$, $S\in L_x^2H_v^{-1}(\D)$ and $c \in \R$. A \emph{weak solution to \eqref{eq:kolm_both} on $\D$} is a function $f \in L^2_xH^1_v(\D)$ that, for every $\varphi\in C^1_c(\D)$, satisfies
    \begin{equation}\label{eq:kolm_rad_weak_torus}
        \langle S , \varphi \rangle
        = \int_{\D} \left[{\nabla}_vf \cdot {\nabla}_v\varphi - f(v\cdot \nabla_x \varphi) + c f \varphi \right] dz.
    \end{equation}
\end{problem}

The \emph{inflow boundary problem} is handled by Problem \ref{prob:trace_G}, which is well-posed for $c\geq0$ (Theorem \ref{th:wellpos_inflow_initial}). The \emph{specular reflection boundary problem} is described by Problem \ref{prob:refl}{} (without trace) and by Problem \ref{prob:trace_Ro} (with trace); similarly, the \emph{periodic problem on $\torus$} is described by Problem \ref{prob:torus}{} (without trace) and Problem \ref{prob:trace_To} on $\Omega =(0,1)^d$ (with trace). 
The reason behind the double formalisation of the reflection problem and the periodic problem is easily explained: the versions without trace are well-posed for $c\geq0$ (Theorem \ref{th:wellpos_refl_initial}) and they are the natural formulations for both problems; however, they do not carry any information on the $L^2(\Gamma, |n_x \cdot v|)$ norm of the solution at the boundary, which is recovered for essentially bounded solutions through the formulations with trace (Proposition \ref{th:extend_to_boundary_BC}).

\begin{remark}\label{rm:sol_boundary_probl}
    If $f$ is a weak solution in $\D$ up to the boundary in the sense of Definition \ref{def:up_to_boundary}, it can be seen as a solution to Problem \ktr{($\Bo_{f_\Gamma}$)}. Moreover, a solution to Problem \ref{prob:trace_Ro} or \ref{prob:trace_To} is, in particular, a solution to \ref{prob:refl}{} or \ref{prob:torus}{} respectively.
\end{remark}

\subsection{Main results}
\label{sec:main_results}
Let us introduce the following hypoelliptic spaces.
\begin{definition}\label{def:Hyp}
    Let $\Omega$ be either a bounded Lipschitz domain or $\torus$. Define the space
    \begin{equation*}
    \Hyp(\D) 
    \coloneqq \left\{ f \in L^2_xH^1_v(\D) \mid v \cdot \nabla_xf \in L^2_xH^{-1}_v(\D)\right\},
\end{equation*}
endowed with the norm
$$\|f\|_{\Hyp(\D)} \coloneqq \|f\|_{L^2(\D)} + \|{\nabla}_v f\|_{L^2(\D)} + \|v \cdot \nabla_xf\|_{L^2_xH^{-1}_v(\D)}.$$
\end{definition}

For clarity, let us specify that the transport operator $v \cdot \nabla_x f$ is defined in a dual sense on $C^1_c(\D)$, and the assumption $v \cdot \nabla_xf \in L^2_xH^{-1}_v(\D)$ means that it enjoys a unique extension on $L^2_xH^1_v(\D)$, where the uniqueness is due to the density of $C^1_c(\D)$ in $L^2_xH^1_v(\D)$.

For a bounded $\Omega$, it is immediate to see that weak solutions to \eqref{eq:kolm_both} in the interior of $\D$ (Definition \ref{def:interior}) belong to $\Hyp(\D)$. Indeed, if $f\in L^2_xH^1_v(\D)$ is such a solution, from the weak formulation \eqref{eq:kolm_rad_weak_interior} we get
\begin{equation} \label{eq:explain_bounded}
\begin{aligned}
    \left| \langle v \cdot \nabla_x f, \varphi \rangle \right|
    &=
    \left|\langle S , \varphi \rangle - \int_{\D} \left[{\nabla}_vf \cdot {\nabla}_v\varphi +c f \varphi \right] dz \right| \\
    &\leq \left( \|S\|_{L^2_xH^{-1}_v(\D)} + |c|\|f\|_{L^2(\D)} + \|{\nabla}_vf\|_{L^2(\D)} \right) \|\varphi\|_{L^2_xH^1_v(\D)}
\end{aligned}
\end{equation}
for $\varphi\in C^1_c({\D})$, and $C^1_c({\D})$ is dense in $L^2_xH^{1}_v(\D)$. Similarly, for $\Omega=\torus$, solutions to Problem \ref{prob:torus}{} lie in $\Hyp(\D)$ (same argument from the weak formulation \eqref{eq:kolm_rad_weak_torus}).

\begin{definition}\label{def:hypA}
    Let $\Omega$ be a bounded Lipschitz domain. The space
    $$\Hyptr(\D)$$
    consists of functions $f\in\Hyp(\D)$ for which there exists a function $f_{\Gamma} \in L^2(\Gamma, |n_x \cdot v|)$,
    called \emph{trace} of $f$, such that $v\cdot\nabla_xf$ satisfies
    \begin{equation}\label{eq:transA}
    \langle v \cdot \nabla_xf, \varphi  \rangle
    = \int_\Gamma (n_x \cdot v) f_{\Gamma} \varphi \ d\sigma - \int_\D f ( v\cdot \nabla_x \varphi) \ dz,
    \qquad \varphi\in C^1_c(\overline{\D}).
\end{equation}
It is endowed with the norm
\begin{equation*}
    \|f\|_{\Hyptr(\D)} \coloneqq 
    \|f\|_{L^2(\D)} + \|{\nabla}_v f\|_{L^2(\D)}  + \|v\cdot\nabla_xf\|_{L^2_xH^{-1}_v(\D)} + \|f_{\Gamma}\|_{L^2(\Gamma,|n_x \cdot v|)}.
\end{equation*}
Define also the subspaces of functions whose trace vanishes on $\Gamma_-$ or $\Gamma _+$, respectively:
\begin{equation*}
    \Hyptrpm(\D)
    \coloneqq
    \left\{ f \in \Hyptr(\D) \mid f_{\Gamma}=0 \ \text{on} \ \Gamma_\mp \right\}.
\end{equation*}
\end{definition}

For $f\in \Hyptr(\D)$ its trace $f_{\Gamma}$ is unique, since the traces of $C^1_c(\overline{\D})$ are dense in $L^2(\Gamma)$. The assumptions $v\cdot\nabla_xf\in L^2_xH^{-1}_v(\D)$ and \eqref{eq:transA} impose that $v\cdot \nabla_xf$ extends to $\varphi\in L^2_xH^1_v(\D)$ in a way that agrees with integration by parts when $\varphi\in C^1_c(\overline{\D})$.

Weak solutions to \eqref{eq:kolm_both} in $\D$ up to the boundary (Definition \ref{def:up_to_boundary}) lie in $\Hyptr(\D)$; in particular, for any boundary operator $\Bo$, solutions to Problem \ref{prob:trace} lie in $\Hyptr(\D)$. Indeed, let $f\in L^2_xH^1_v(\D)$ be a solution up to the boundary, with trace $f_\Gamma\in L^2(\Gamma,|n_x \cdot v|)$, and consider the linear operator $T_{f,f_\Gamma}:C^1_c(\overline{\D})\to \R$ defined by the right-hand side of \eqref{eq:transA}. Exactly as shown in \eqref{eq:explain_bounded}, from the weak formulation \eqref{eq:kolm_rad_weak} we get
\begin{align*}
    \left| \langle T_{f,f_\Gamma}, \varphi \rangle \right|
    &\leq \left( \|S\|_{L^2_xH^{-1}_v(\D)} + |c|\|f\|_{L^2(\D)} + \|{\nabla}_vf\|_{L^2(\D)} \right) \|\varphi\|_{L^2_xH^1_v(\D)}
\end{align*}
for $\varphi\in C^1_c(\overline{\D})$. Since $C^1_c(\overline{\D})$ is dense in $L^2_xH^{1}_v(\D)$, we have $T_{f,f_\Gamma}\in L^2_xH^{-1}_v(\D)$. Moreover, $f$ is also a solution in the interior of $\D$, and therefore $f\in\Hyp(\D)$ and $v \cdot \nabla_xf \in L^2_xH^{-1}_v(\D)$. The two operators $T_{f,f_\Gamma}$ and $v \cdot \nabla_xf$ coincide on $C^1_c(\D)$, and thus on $L^2_xH^1_v(\D)$ by density and boundedness. Therefore, $v \cdot \nabla_xf$ satisfies \eqref{eq:transA}.

Our two main results are a one-sided trace estimate and a Poincaré-type inequality for functions in $\Hyptr(\D)$. They are proved in Sections \ref{sec:trace_theorem} and \ref{sec:poincare} respectively.

\begin{theorem}[One-sided trace estimate for $\Hyptr(\D)$] \label{th:strong_trace_0}
    Let $\Omega$ be a {$C^{1,\alpha_V}$} bounded domain. Then, for $f\in\Hyptr(\D)$, it holds
    \begin{gather}
    \label{eq:trace_partial1}
        \|f_\Gamma\|_{L^{2}(\Gamma_+, |n_x\cdot v|)} \lesssim  \|f\|_{\Hyp(\D)} +\|f_\Gamma\|_{L^{2}(\Gamma_-, |n_x\cdot v|)}, \\
    \label{eq:trace_partial2}
        \|f_\Gamma\|_{L^{2}(\Gamma_-, |n_x\cdot v|)} \lesssim  \|f\|_{\Hyp(\D)} +\|f_\Gamma\|_{L^{2}(\Gamma_+, |n_x\cdot v|)}.
    \end{gather}
    In particular, the trace operators for $\Hyptrp(\D)$ and $\Hyptrm(\D)$ respectively
    \begin{gather*}
        (-)_{\Gamma} \colon \Hyptrpm(\D) \to L^{2}(\Gamma, |n_x\cdot v|), \quad f \to f_\Gamma
    \end{gather*}
    are bounded.
\end{theorem}

\begin{theorem}[Poincaré inequality for $\Hyptr(\D)$] \label{th:poincare1}
    Let $\Omega$ be a bounded Lipschitz domain such that $\partial\Omega$ is $C^1$ at least at one point. Then, for $f \in \Hyptr(\D)$, it holds
    \begin{equation}\label{eq:poincare}
        \|f \|_{L^2(\D)} \lesssim
        \|{\nabla}_v f\|_{L^2(\D)}  + \|v\cdot\nabla_xf\|_{L^2_xH^{-1}_v(\D)} + \|f_\Gamma\|_{L^{2}(\Gamma_-, |n_x \cdot v|)}.
    \end{equation}
    In particular,
    \begin{equation}
    \label{eq:zz_mean_poincare}
        |\langle f \rangle_\D| \lesssim \|{\nabla}_{v} f\|_{L^{2}(\D)} + \|v\cdot\nabla_xf \|_{L^{2}_x H^{-1}_v(\D)} + \|f_\Gamma\|_{L^{2}(\Gamma_-, |n_x \cdot v|)} .
    \end{equation}
\end{theorem}

Our third main result is well-posedness for the inflow boundary value problem; it is proved in Section \ref{sec:well_inflow}.

\begin{theorem}[Well-posedness of the inflow boundary value problem]\label{th:wellpos_inflow_initial}
    Let $\Omega$ be a {$C^{1,\alpha_V}$} bounded domain. Let {$c\geq0$}, $S\in L_x^2H_v^{-1}(\D)$ and $G \in L^{2}(\Gamma_-, |n_x\cdot v|)$. Then, there exists a unique solution $f$ to Problem \ref{prob:trace_G}. It satisfies $f\in\Hyptr(\D)$ and
    \begin{equation}\label{eq:energy1}
        \|f\|_{\Hyptr(\D)} \lesssim_c \|S\|_{L^2_xH_v^{-1}(\D)} + \|G\|_{L^{2}(\Gamma_-, |n_x\cdot v|)}.
    \end{equation}
\end{theorem}

Turning to Problem \ref{prob:refl}{}, consider the following function space.
\begin{definition}
    Let $\Omega$ be a bounded Lipschitz domain. The space
    $$\Hyprefl(\D)$$
    consists of functions $f\in\Hyp(\D)$ such that $v\cdot\nabla_xf$ satisfies
    \begin{equation}\label{eq:transB}
    \langle v \cdot \nabla_xf, \varphi  \rangle
    = - \int_\D f ( v\cdot \nabla_x \varphi) \ dz,
    \qquad \varphi\in \Crefl.
\end{equation}
It is endowed with the norm $\| \cdot \|_{\Hyp(\D)}$.
\end{definition}

Again, the assumptions $v\cdot\nabla_xf\in L^2_xH^{-1}_v(\D)$ and \eqref{eq:transB} mean that $v\cdot \nabla_xf$ extends to $\varphi\in L^2_xH^1_v(\D)$ in a way that agrees with integration by parts for $\varphi\in \Crefl$. 

Weak solutions to Problem \ref{prob:refl}{} belong to $\Hyprefl(\D)$. Indeed, if $f\in L^2_xH^1_v(\D)$ is weak solution, consider the linear operator $T_{f}:\Crefl\to \R$ defined by the right-hand side of \eqref{eq:transB}. Exactly as shown in \eqref{eq:explain_bounded}, from the weak formulation \eqref{eq:kolm_rad_weak_refl} and the fact that the spaces $C^1_c(\D)\subseteq\Crefl$ are dense in $L^2_xH^{1}_v(\D)$, we get that $T_{f}\in L^2_xH^{-1}_v(\D)$. Since $f$ is also a solution in the interior of $\D$ and therefore $v \cdot \nabla_xf \in L^2_xH^{-1}_v(\D)$, the fact that $T_{f}$ and $v \cdot \nabla_xf$ coincide on $C^1_c(\D)$ yields that they coincide on $L^2_xH^1_v(\D)$. Therefore, $v \cdot \nabla_xf$ satisfies \eqref{eq:transB}.

Well-posedness for Problem \ref{prob:refl}{} is our fourth main result. We state it together with Problem \ref{prob:torus}{} (although the latter was already solved in \cite{Albritton2024VariationalEquation}) to highlight the common structure. The proof can be found in Section \ref{sec:refl_torus}.

\begin{theorem}[Well-posedness with specular reflection and on $\torus$] \label{th:wellpos_refl_initial}
    Let $\Omega$ be a {$C^{1,\alpha_V}$} bounded domain (or $\torus$). Let $c>0$, and $S\in L^2_xH_v^{-1}(\D)$. Then there exists a unique solution $f$ to Problem \ref{prob:refl}{} (resp. Problem \ref{prob:torus}{}). It satisfies $f\in\Hyprefl(\D)$ (resp. $f\in\Hyp(\D)$) and
    \begin{equation}\label{eq:energy6.1}
    \|f\|_{\Hyp(\D)}\lesssim_c \|S\|_{L^2_xH_v^{-1}(\D)}.
    \end{equation}
    If $c=0$ and $\langle S\rangle_\D=0$, there exists a solution: it is unique within the normalised class $\langle f\rangle_\D=0$, and satisfies \eqref{eq:energy6.1}.
\end{theorem}

We recover the behaviour of solutions to Problems \ref{prob:refl}{} and \ref{prob:torus}{} at the boundary when $f\in L^\infty(\D)$. The proof is again in Section \ref{sec:refl_torus}.

\begin{proposition}\label{th:extend_to_boundary_BC}
    Let $\Omega$ be a {$C^{1,\alpha_V}$} bounded domain, and let $f$ be the solution to Problem \ref{prob:refl}{} with $c>0$. If $f\in L^\infty(\D)$, then there exists $f_\Gamma\in L^\infty(\Gamma)$ such that $f$ is the solution to Problem \ref{prob:trace_Ro} on $\D$ with trace $f_\Gamma$.

    The same holds for $\Omega =\torus$ and Problem \ref{prob:torus}{} with $c>0$, with respect to the fundamental domain ${\D}\coloneqq (0,1)^d \times {V}$: if $f\in L^\infty(\D)$, there exists $f_{\partial{\D}} \in L^\infty(\partial{\D})$ such that $f|_{{\D}}$ is the solution to Problem \ref{prob:trace_To} on ${\D}$ with trace $f_{\partial{\D}}$.
\end{proposition}

As a consequence of Theorem \ref{th:wellpos_refl_initial} and Proposition \ref{th:extend_to_boundary_BC}, we obtain existence and uniqueness of the solution to Problems \ref{prob:trace_Ro} and \ref{prob:trace_To}, with $c>0$ and $S\in L^\infty(\D)$ -- see Proposition \ref{prop:refl_torus_trace}. We also give a collection of weak maximum principles for the three problems in Proposition \ref{th:max_princ}.

\subsection{Review of the literature}
\label{sec:literature}

The existence of solutions for the Kolmogorov-Fokker-Planck equation on bounded domains was addressed relatively recently. Initially, Carrillo \cite{Carrillo1998GlobalSystem} established an existence theory for the kinetic Vlasov-Poisson-Fokker-Planck system on bounded domains with Gaussian velocities, under absorbing and reflection-type boundary conditions, with the aid of a representation theorem by J.-L. Lions \cite{Lions1957SurCylindriques} known as Lions-Lax-Milgram (see Theorem \ref{th:lions}). Regarding well-posedness, Carrillo relied on the trace theorem \cite[Lemma 2.3]{Carrillo1998GlobalSystem}, although later Albritton, Armstrong, Mourrat, and Novack \cite{Albritton2024VariationalEquation} pointed out the incompleteness of its proof. They left as open questions the existence of a trace operator from $\Hyp(\D)$ to $L^2(\Gamma, |v \cdot n_x|)$, that is
\begin{equation}\label{eq:trace_problem}
    \|f\|_{L^2(\Gamma, |v \cdot n_x|)} \leq C_\Omega \|f\|_{\Hyp(\D)},
    \qquad f\in C_c^\infty(\overline{\D})
\end{equation}
for some constant $C_\Omega>0$ \cite[Question 1.8]{Albritton2024VariationalEquation}, and the related solvability of the kinetic Fokker-Planck equation with boundary.

In \cite{Albritton2024VariationalEquation}, the authors addressed the time-dependent and stationary Kramers-Fokker-Planck equations with bounded position and unbounded velocity with Gaussian weight. They introduced a hypoelliptic Poincaré inequality for functions in $\Hyp(\D)$ for either $\Omega$ a bounded Lipschitz domain or $\Omega=\torus$ \cite[Theorem 1.3]{Albritton2024VariationalEquation}, and developed a variational approach to obtain the weak solution on $\torus$ in two ways: on one hand, by the Lions-Lax-Milgram theorem combined with their hypoelliptic Poincaré inequalities (see Section \ref{sec:comm_poincare}); on the other hand, as the minimiser of a uniformly convex functional. The functional approach was later adapted to the Kolmogorov equation on bounded domains with inflow boundary conditions and rough coefficients in \cite{Litsgard2021TheCoefficients}, but the trace problem was not addressed therein. An alternative approach towards the existence of weak solutions on domains without boundary can be found in \cite{Avelin2024AExpansions}, where a Galerkin-type approximation for solutions on $\torus$ or $\R^d$ is developed. 

Then, Zhu \cite{Zhu2025RegularityDomains} introduced the notion of solution-trace pairs (see Section \ref{sec:trace}) for the equation in the bounded-velocity model. He established existence (through a vanishing viscosity method), uniqueness and energy estimates (by renormalisation techniques) of the solution to the time-dependent kinetic Fokker-Planck equation, under inflow boundary condition and specular reflection condition with a bounded source. Finally, Avelin and Hou \cite{Avelin2026WeakDomains} adapted the work of Zhu to show existence and uniqueness of the solution to the stationary kinetic Fokker-Planck equation with bounded position and velocity. They also compared the Lions-Lax-Milgram approach and the vanishing viscosity method -- see Remark \ref{rm:comm_avelin}. Still, in \cite{Avelin2026WeakDomains}, energy estimates for the stationary equation were not addressed.

Our work gives a unified approach to the inflow boundary problem, the specular reflection problem and the periodic problem on $\torus$ for the stationary Kolmogorov equation with bounded position and spherical or Gaussian velocities. In particular, we build on the strategy of \cite{Albritton2024VariationalEquation} to get the solution to the equation without damping ($c=0$) as the limit of the solutions of the penalised equations with damping ($c>0$). Indeed, the Lions-Lax-Milgram theorem and the vanishing viscosity method (which are standard techniques to prove the existence of solutions) apply for $c>0$ (see Lemma \ref{th:existence1}). Moreover, while energy estimates for the time-dependent Kolmogorov-Fokker-Planck equation can be obtained even in the absence of damping ($c=0$) using the Gr\"onwall inequality (see \cite[Lemma 2.7]{Zhu2025RegularityDomains}), this is not the case for the stationary equation. The key point in our work is the estimate \eqref{eq:zz_mean_poincare} on the average of functions in $\Hyptr(\D)$: together with the Poincaré inequality from \cite{Albritton2024VariationalEquation} (which is enough to solve the problem on $\torus$ only), it yields the Poincaré-type estimate \eqref{eq:poincare}. The latter and the renormalisation formula (Lemma \ref{th:normalisation_kolm}) are enough to establish energy estimates which are stable for $c\geq 0$, which allow us to perform the limit $c\to 0$.

Since the first dissemination of the present paper, the related work \cite{Niebel2026SharpTheory}, carried out jointly with Lukas Niebel, has established a full kinetic trace theory, with sharp positive and negative results. In particular, for the Gaussian velocity model $V=\Rg$, the trace estimate \eqref{eq:trace_problem}, and therefore the existence of a trace operator with natural weight, fails on any $C^{1,1}$ domain $\Omega$ in dimension $d \geq 2$. For the spherical velocity model $V=\Sd$, the trace estimate \eqref{eq:trace_problem} in general does not hold in the class of $C^{1,\alpha}$ domains for any $\alpha < \frac{1}{2}$, whereas it holds for $C^{1,1/2}$ domains. In this work, we give well-posedness results despite the (possible) failure of the natural trace estimate.

\subsubsection{The kinetic trace problem}
\label{sec:trace}

The difficulty of the trace estimate \eqref{eq:trace_problem} is related to the possibly wild behaviour of the trace around the grazing set $\Gamma_0$, due to the nature of the transport operator $v \cdot \nabla_x$. 
Let us remark that the power $1$ for the weight $|v \cdot n_x|$ is the natural one: due to the transport operator $v \cdot \nabla_x$, it naturally appears in the weak formulation \eqref{eq:kolm_rad_weak}, where the boundary term is interpreted as a scalar product.

The trace problem was solved positively for $d=1$ by Baouendi and Grisvard in \cite[Theorem 1]{Baouendi1968SurType}. Their argument relies on the auxiliary function ${f}(x,v) \coloneqq f(x,|v|)$ and estimates the $L^2_xH^{-1}_v$-norm of $|v|\partial_x{f}$ with the $L^2_xH^{-1}_v$-norm of $v\partial_xf$: together with integration by parts, this is enough to show \eqref{eq:trace_problem}. The same strategy in higher dimensions fails, as well described in \cite[Appendix A]{Armstrong2019VariationalEquation}. The key point is that an estimate on $v\cdot \nabla_xf$ in dimension $d=1$ is in fact an estimate on the full gradient of $\nabla_xf$, which degenerates on a hyperplane, while in higher dimensions it is only an estimate of one particular projection of $\nabla_xf$ at any given point.

For the Gaussian velocity model $V=\Rg$, in \cite{Armstrong2019VariationalEquation} Armstrong and Mourrat proved that the trace operator is bounded away from the grazing set on $C^{1,1}$ domains. 

\begin{lemma}[Armstrong, Mourrat {\cite[Lemma 4.3]{Armstrong2019VariationalEquation}}] \label{th:trace_away}
    Let $\Omega$ be a $C^{1,1}$ bounded domain, and $\D \coloneqq \Omega \times \Rg$. Then, for any compact set $K \subseteq \Gamma \setminus \Gamma_0$, the trace operator from $C^\infty(\D)$ to $L^2(K,|v \cdot n_x|)$ extends to a continuous linear operator on $\Hyp(\D)$. This yields a pointwise definition of trace $f|_\Gamma \in L^2_{\loc}(\Gamma \setminus \Gamma_0, |v \cdot n_x|)$ for $f\in \Hyp(\D)$.
\end{lemma}

Armstrong and Mourrat also claimed the one-sided trace estimate of Theorem \ref{th:strong_trace_0} in \cite[Lemma 4.5]{Armstrong2019VariationalEquation}, although Brunken and Smetana \cite[Supplement material, Section SM2]{Brunken2022StableEquation} pointed out the incompleteness of the proof.

Then, Silvestre obtained a trace theorem in $L^2(\Gamma, \omega_2)$ -- that is, with trace weight $\omega_2(x,v) = \min\{|v \cdot n_x|,|v \cdot n_x|^2\}$ -- in \cite[Proposition 4.3]{Silvestre2022HolderBoundary}, again for $C^{1,1}$ domains. 

\begin{proposition}[Silvestre {\cite[Proposition 4.3]{Silvestre2022HolderBoundary}}]
\label{th:trace_2}
    Let $\Omega$ be a bounded $C^{1,1}$ domain and $V =\Rg$. Then boundary restriction extends to a bounded trace operator
    \begin{equation*}
        \tr_x: \Hyp(\D) \to L^2(\Gamma,\omega_2).
    \end{equation*}
\end{proposition}

See \cite[Proposition 7.8]{Niebel2026SharpTheory} for a proof. The proof first approximates $f\in \Hyp(\D)$ with functions $f_n\in C^\infty_c(\D)$
\footnote{\label{note:smooth}%
See for instance \cite[Proposition 2.2 in either version]{Armstrong2019VariationalEquation,Albritton2024VariationalEquation}, \cite[Lemma 4.2]{Silvestre2022HolderBoundary}, and \cite[Lemma 5.2]{Avelin2026WeakDomains}.
}%
, and then applies integration by parts to $v\cdot\nabla_x f_n$ against $(v \cdot n(x))f_n$, where $n(x)$ is an outward normal vector field defined in a neighbourhood of $\Gamma$ and coincides with $n_x$ on $\Gamma$. This is enough to recover the weight $|v\cdot n_x|^2$, but requires some regularity for $n(x)$. 
In fact, in \cite[Theorem 4.2]{Niebel2026SharpTheory} it is shown that, for $d\geq 2$ and $V=\Rg$, the trace estimate with weight $\omega_p(x,v) = \min\{|n_x \cdot v|,|n_x \cdot v|^p\}$ for $p\in[1,2)$ fails on every bounded $C^{1,1}$ domain. Therefore, Silvestre's trace theorem is sharp both in the parameter $p$ and in the H\"older regularity of the domains.

For the spherical velocity model $V=\Sd$, the trace problem was also fully solved in \cite{Niebel2026SharpTheory}. In dimension $d=1$, the trace estimate with trace weight $\omega_p$ holds for any $p\in[1,+\infty)$. In dimension $d\geq2$, it holds for $p\in[1,+\infty)$ on every bounded $C^{1,\alpha}$ domain with $\alpha \geq 1/(p+1)$, and fails for every $0<\alpha<1/(p+1)$ on some strictly convex bounded domain of exact regularity $C^{1,\alpha}$. The same regularity threshold holds for the confined-velocity model $V=\V$ too. In particular, the natural trace ($p=1$) has the regularity threshold $C^{1,1/2}$.

\begin{proposition}[Niebel, Valentini {\cite[Proposition 7.9]{Niebel2026SharpTheory}}]
    \label{th:trace_p}
    Let \(1\le p<\infty\) and let \(1/(p+1)\le\alpha\le1\). Let \(\Omega\subset\R^d\) be a bounded \(C^{1,\alpha}\) domain. Let $V$ be either $\Sd$ or a bounded Lipschitz domain $\V\subseteq \R^d$.
    Then boundary restriction extends to a bounded trace operator  
      \[
        \tr_x\colon \Hyp(\D) \to L^2(\partial\Omega\times V,
        \omega_p).
      \]
\end{proposition}

\subsubsection{The functional framework for the inflow problem}

Before the full resolution of the trace problem, Zhu \cite{Zhu2025RegularityDomains} introduced the notion of solution pairs
to define boundary value problems for the Fokker-Planck equation in the absence of a trace operator: a pair function-trace is a solution if it satisfies the weak formulation of the equation up to the boundary, that is, against $C^1_c(\overline{\D})$. Continuity is not assumed in the correspondence function-trace. This approach proves relevant in light of recent results in \cite{Niebel2026SharpTheory} showing that the natural trace problem indeed fails for a non-negligible class of spatial domains (for $V=\Rg$, all domains with regularity $C^{1,1}$; for $V=\Sd$, some strictly convex domains with regularity lower than $C^{1,1/2}$).

Solution pairs were initially used by Mischler in \cite{Mischler1999OnEquation} for the $L^1$ solutions to Vlasov equations, and then generalised to other kinetic equations in \cite{Mischler2010KineticConditions}. Also, if a pair function-trace is the solution to a given equation, it can be handled jointly thanks to the renormalisation methods developed by DiPerna and Lions \cite{Diperna1973GlobalEquations,DiPerna1989OrdinarySpaces}, including the boundary term in the renormalisation.

In this paper, Problem \ref{prob:trace} is modelled on Zhu's notion of solution pairs. Accordingly, we introduced the spaces $\Hyptr(\D)$ and $\Hyprefl(\D)$, which provide a suitable functional framework to solve the problems. In fact, we show that every $f\in\Hyptr(\D)$ is a weak solution up to the boundary of an equation of the form \eqref{eq:kolm_both} -- see Section \ref{sec:trace_theorem}.

\subsubsection{Energy estimates: renormalisation and Poincaré inequalities}
\label{sec:comm_poincare}

The weak formulation of the Kolmogorov equation  \eqref{eq:kolm_rad_weak} on a bounded domain $\Omega$ does not enjoy enough continuity to be tested for the solution itself. In particular, integration by parts of functions in $\Hyp(\D)$ against the transport operator $v\cdot \nabla_x$ is not justified a priori, even if $\Omega$ is $C^{1,1}$ and a boundary function is locally defined (Lemma \ref{th:trace_away}). As explained in Section \ref{sec:main_results}, if $f$ is a weak solution to \eqref{eq:kolm_both} in $\D$ up to the boundary, the weak formulation itself implies that integration by parts for $v\cdot \nabla_x$ holds, but only against test functions in $C^1_c(\overline{\D})$. 
This is the main obstacle in proving energy estimates directly from the equation -- see \cite[Remark 4.3]{Albritton2024VariationalEquation}. 

To overcome the obstacle, the strategy followed in \cite{Armstrong2019VariationalEquation} (the first version of \cite{Albritton2024VariationalEquation}) consisted of approximating solutions with smooth functions vanishing in a neighbourhood of the grazing set $\Gamma_0$ \cite[Lemma 4.5]{Armstrong2019VariationalEquation}, and then justifying integration by parts just for the latter \cite[Proposition 4.4]{Armstrong2019VariationalEquation}. However, Brunken and Smetana \cite[Supplement material, Section SM2]{Brunken2022StableEquation} pointed out the incompleteness of the proof of the density result. 

Still, following techniques of Silvestre \cite{Silvestre2022HolderBoundary} and Zhu \cite{Zhu2025RegularityDomains}, energy estimates for solutions to Problems \ref{prob:trace}, \ref{prob:refl}{} and \ref{prob:torus}{} with $c>0$ can be achieved through a suitable renormalisation formula (Lemma \ref{th:normalisation_kolm}). However, the estimates obtained in this way (see inequalities \eqref{eq:zz_with_c} and \eqref{eq:zz_with_c2}) blow up in the limit $c\to 0$. Then, to find a solution to the three problems for $c=0$, the key step is proving estimates that are stable in the limit $c\to 0$ using the Poincaré-type inequality \eqref{eq:poincare}, following a strategy that was carried out in \cite{Albritton2024VariationalEquation} for $\Omega=\torus$. Hypoelliptic Poincaré inequalities for $\Hyp(\D)$, both when $\Omega$ is a bounded Lipschitz domain and when $\Omega =\torus$, were given in \cite[Theorem 1.3]{Albritton2024VariationalEquation}, and we recall them in Theorem \ref{th:poincare2}. Actually, in \cite[Theorem 1.3]{Albritton2024VariationalEquation} the inequality \eqref{eq:poincare} was already proved for the following space (we use the notation from \cite{Albritton2024VariationalEquation}):
\begin{equation}
\label{eq:trace0_armstrong}
    \HypO(\D) \coloneqq \overline{\{f \in C_c^\infty(\overline{\D}) \mid f|_{\Gamma_-}=0 \}}^{\Hyp(\D)}.
\end{equation}
Notice that $\HypO(\D) \subseteq \Hyptrp(\D)$ and the inclusion, a priori, could be proper. In Theorem \ref{th:poincare1}, we extend to $\Hyptr(\D)$ the proof of \eqref{eq:poincare} that was given in \cite[Theorem 1.3]{Albritton2024VariationalEquation} for $\HypO(\D)$ -- see Remark \ref{rm:differences_poincare} for more details.

\subsubsection{Specular reflection}
Two weak formulations for the specular reflection problem can be found in the literature: the one given in Problem \ref{prob:refl}{}, used for instance by Silvestre \cite{Silvestre2022HolderBoundary}, and the one given by Problem \ref{prob:trace_Ro}, used by Zhu \cite{Zhu2025RegularityDomains}.

For the latter, well-posedness can be obtained for a source $S\in L^\infty(\D)$ (Proposition \ref{prop:refl_torus_trace}): a direct proof following ideas of Zhu \cite[Corollary 2.11]{Zhu2025RegularityDomains} is recalled in Appendix \ref{app:auxiliary}. The key point is that, under specular reflection, the boundary term
$$\int_{\Gamma} (n_x \cdot v)f_{\Gamma}^2 \,d\sigma$$
vanishes: therefore, no information on the behaviour of the solution at the boundary is naturally carried by the problem. Indeed, in Zhu's approach, a solution pair to Problem \ref{prob:trace_Ro} is found through a weak-$*$ limit using the weak maximum principle given in Proposition \ref{th:max_princ}\ref{eq:energy_ART}, which indeed is helpful only when the source is essentially bounded.

On the contrary, the formulation without trace given in Problem \ref{prob:refl}{} satisfies a full well-posedness result in $L^2$, not just for essentially bounded sources (Theorem \ref{th:wellpos_refl_initial}). Moreover, it also recovers Proposition \ref{prop:refl_torus_trace} as a straightforward corollary, as shown in its proof in Section \ref{sec:refl_torus}.

\section{Function setting and core a priori estimates}
\label{sec:hypo_setup}

Consider either $V=\Sd$ or $V=\Rg$.

\subsection{Function spaces}
First, let $\Omega$ be a bounded Lipschitz domain.
As solutions to Problems \ref{prob:trace} and \ref{prob:refl}{} are in particular solutions in the interior of $\D$, they also lie in $\Hyp(\D)$. However, the property $v \cdot \nabla_xf \in L^2_xH^{-1}_v(\D)$ is not enough to encode the behaviour of $f$ at the boundary; that is, it is not guaranteed that, when evaluated on $\varphi\in C^1_c(\overline{\D})$ or $\varphi\in \Crefl$, the operator $v \cdot \nabla_xf$ satisfies \eqref{eq:transA} or \eqref{eq:transB} respectively. For this reason, we have introduced the solution spaces $\Hyptr(\D)$ and $\Hyprefl(\D)$, in which functions are such that their transport $v \cdot \nabla_xf$ satisfies the dual formulation up to the boundary.

Regarding the inflow boundary problem, the trace terminology and notation introduced in Problem \ref{prob:trace} and in the space $\Hyptr(\D)$ are coherent, since solutions to Problem \ref{prob:trace} lie in $\Hyptr(\D)$. In particular, notice how the definition of the latter combines the notion of solution pairs from \cite{Zhu2025RegularityDomains} with the definition of $\Hyp(\D)$ from \cite{Albritton2024VariationalEquation}.

Concerning specular reflection, we introduced Problem \ref{prob:refl}{} as a problem without boundary because, to some extent, we are working on the quotient manifold obtained from $\D$ via the identification $(x,v) \sim (x, v-2(n_x \cdot v) n_x)$. Also, in the definition of $\Hyprefl(\D)$, the reflection symmetry required at the boundary becomes a geometric constraint imposed at the level of function spaces, similarly to what happens on $\torus$.

\begin{remark}
    Observe that, for $V=\Sd$, the weak formulations of Section \ref{sec:problems} hold for $\varphi\in H^1_0(\D)$ and $H^1(\D)$ respectively (with the assumption $\Tr_x \varphi = \Bo_\Ro (\Tr_x \varphi)$ for problem \ref{prob:refl}). This is not the case for $V=\Rg$, as the term $\int_\D f (v \cdot \nabla_x \varphi) \, dz$ might be unbounded.
\end{remark}

\subsection{Boundary regularity}
\label{sec:boundary_regularity}
Although the results in this section do not assume a stronger regularity for $\Omega$, the renormalisation formula (Lemma \ref{th:normalisation_kolm}) later used to prove energy estimates requires $C^{1,\alpha_V}$ domains. As mentioned before, this is to ensure that the sub-optimal trace operator $\tr_x$ with trace weight $\omega_2(x,v) = \min\{|n_x \cdot v|,|n_x \cdot v|^2\}$ is bounded (Propositions \ref{th:trace_2} and \ref{th:trace_p}). This is enough to show that, when $\Omega$ is a $C^{1,\alpha_V}$ bounded domain, for $f\in \Hyp(\D)$ and $\varphi \in C_c^1(\overline{\D})$, it holds
\begin{equation}\label{eq:green}
    \langle v \cdot \nabla_x f, \varphi\rangle
    = - \int_{\D} f (v \cdot \nabla_x \varphi) \ dz + \int_{\Gamma} (v \cdot n_x) \tr_x(f) \varphi \ d\sigma.
\end{equation}
This can be shown by approximating $f$ in $C^\infty_c(\D)$ and applying integration by parts.
Consequently, when $\Omega$ is a $C^{1,\alpha_V}$ bounded domain, the function spaces reduce to
\begin{gather}
\label{eq:hyptr2}
\Hyptr(\D) 
= 
\left\{ f\in \Hyp(\D) \mid
    \tr_xf \in L^2(\Gamma, |n_x \cdot v|) \right\},
\\
\label{eq:hyprefl2}
\Hyprefl(\D) 
= 
\left\{ f\in \Hyp(\D) \mid
    \tr_xf \in L^2(\Gamma, |n_x \cdot v|),
    \tr_x f = \Bo_\Ro \tr_xf \ \text{on} \ \Gamma_-\right\}.
\end{gather}
Equality \eqref{eq:hyptr2} follows directly from \eqref{eq:green}. Regarding \eqref{eq:hyprefl2}, if $f\in\Hyprefl(\D)$, assumption \eqref{eq:transB} and equation \eqref{eq:green} together give
\[
\int_\Gamma
(n_x\cdot v) \tr_xf \phi \, d\sigma =0
\qquad
\text{for} \ \phi\in \Crefl.
\]
Performing the change of variables $v\mapsto R_xv \coloneqq v-2(n_x \cdot v) n_x$ in the integral over $\Gamma_+$ gives, for every $\phi\in \Crefl$,
\[
\int_{\Gamma_-}
(n_x\cdot v)
\big[
\tr_xf(x,v)- \tr_xf(x,R_xv)
\big]\varphi(x,v) \, d\sigma
=0.
\]
Hence $\tr_xf(x,v)= \tr_xf(x,R_xv)= \Bo_\Ro \tr_x f (x,v)$ for a.e. $(x,v)\in\Gamma_-$. The other inclusion is straightforward.

\subsection{Poincaré inequality}
\label{sec:poincare}

In \cite[Theorem 1.3]{Albritton2024VariationalEquation}, the authors introduced a hypoelliptic Poincaré inequality for functions in $\Hyp(\D)$, with Gaussian velocities $V=\Rg$. The proof adapts straightforwardly to spherical velocities $V=\Sd$, and we include it here for the sake of completeness.

\begin{theorem}\label{th:poincare2}
Let $\Omega$ be either a bounded Lipschitz domain or $\torus$. Then, for any $f \in \Hyp(\D)$, it holds
    \begin{equation}
    \label{eq:poincareD}
        \|f- \langle f \rangle_{\D}\|_{L^2(\D)} \lesssim
        \|{\nabla}_v f\|_{L^2(\D)} + \|v \cdot \nabla_xf\|_{L^2_xH^{-1}_v(\D)}. 
    \end{equation}
\end{theorem}
\begin{proof}
Let $f \in \Hyp(\D)$.
First, applying the Gaussian Poincaré inequality in $v\in\Rg$ or the standard Poincaré inequality in $v\in\Sd$, and integrating in $x\in\Omega$, we get
    \begin{equation}
    \label{eq:zz_stand_poincare}
        \|f-\langle f \rangle_{{V}}\|_{L^2(\D)} \lesssim \|{\nabla}_{v} f\|_{L^2(\D)}.
    \end{equation}
Then, we show that
    \begin{equation}
    \label{eq:zz_der_x_average}
        \|\nabla_x\langle f\rangle_{{V}}\|_{H_x^{-1}(\Omega)} 
        \lesssim \|{\nabla}_{v} f\|_{L^2(\D)}+\|v \cdot \nabla_xf\|_{L^{2}_xH_v^{-1}(\D)}.
    \end{equation}
    Fix some functions $\xi_{1}, \ldots, \xi_{d} \in C_c^{\infty}({V})$ satisfying
    \begin{equation*}
        \int_{{V}} v \xi_{i}(v) \ dv =e_{i},
        \qquad i \in\{1, \dots, d\},
    \end{equation*}
    where $\{e_1, \dots, e_d\}$ is the canonical orthonormal basis for $\R^d$. Then, for each test function $\varphi \in C^1_c(\Omega)$ and for each $i \in\{1, \dots, d\}$, we have
    \begin{align*}
        \int_{\Omega} \partial_{x_{i}} \varphi(x)\langle f\rangle_{{V}}(x) \ dx
        =& \int_{\D} v \cdot \nabla_{x} \varphi(x) \langle f\rangle_{{V}}(x) \xi_{i}(v)  \ dz\\
        =& \int_{\D} v \cdot \nabla_{x} \varphi(x) f(x, v) \xi_{i}(v) \ dz \\
        &+ \int_{\D} v \cdot \nabla_{x} \varphi(x)\big[\langle f\rangle_{{V}}(x)-f(x, v)\big] \xi_{i}(v) \ dz.
    \end{align*}
    By definition of the operator $v \cdot \nabla_xf$ on $C^1_c(\D)$, the first integral is bounded by
    \begin{align} \label{eq:zz_ibp_false}
        \left|\int_{\D} v \cdot \nabla_{x} \varphi(x) f(x, v) \xi_{i}(v) \ dz\right| 
        & = \left| \langle v \cdot \nabla_xf, \varphi \xi_{i} \rangle  \right| \\ 
        &\lesssim \|\varphi \xi_{i}\|_{L^{2}_xH_v^{1}(\D)} \|v \cdot \nabla_{x} f \|_{L^{2}_xH_{v}^{-1}(\D)} \notag \\
        & \lesssim \|\varphi\|_{H^1_x(\Omega)} \|v \cdot \nabla_{x} f \|_{L^{2}_xH_{v}^{-1}(\D)}. \notag
    \end{align}
    To control the second term, we use \eqref{eq:zz_stand_poincare}:
    \begin{align*}
        &\left| \int_{\D} v \cdot \nabla_{x} \varphi(x) \big[\langle f\rangle_{{V}}(x)-f(x, v)\big]\xi_{i}(v)  \ dz\right| \\
        & \hspace{2cm} \lesssim \int_{\D}|v||\xi_{i}(v)||\nabla_{x} \varphi(x)||f(x, v)-\langle f\rangle_{{V}}(x)| \ dz \\
        & \hspace{2cm} \lesssim \|\varphi\|_{H_x^{1}(\Omega)}\|{\nabla}_{v} f\|_{L^{2}(\D)}.
    \end{align*}
    By density of $C^1_c(\Omega)$ in $H^1_0(\Omega)$, we obtain \eqref{eq:zz_der_x_average}.
    Finally, if we prove
    \begin{equation}\label{eq:from_AM}
        \|\langle f\rangle_{{V}}-\langle f \rangle_{\D}\|_{L^{2}(\Omega)}
        \lesssim \|\nabla_x\langle f\rangle_{{V}}\|_{H^{-1}(\Omega)},
    \end{equation}
    then a combination of \eqref{eq:zz_stand_poincare}, \eqref{eq:zz_der_x_average} and \eqref{eq:from_AM} gives \eqref{eq:poincareD}:
    \begin{align*}
        \|f-\langle f \rangle_{\D}\|_{L^{2}(\D)}
        & \leq \|f-\langle f \rangle_{{V}}\|_{L^{2}(\D)} + \|\langle f\rangle_{{V}}-\langle f \rangle _{\D}\|_{L^{2}(\D)} \\
        & \lesssim \|{\nabla}_{v} f\|_{L^{2}(\D)} + \|\nabla_x\langle f\rangle_{{V}}\|_{H^{-1}(\Omega)} \\
        & \lesssim \|{\nabla}_{v} f\|_{L^{2}(\D)} +\|v \cdot \nabla_{x} f\|_{L^{2}_xH^{-1}_v(\D)}.
    \end{align*}
    
    The inequality \eqref{eq:from_AM} is given in \cite[Lemma 3.1]{Albritton2024VariationalEquation}; let us recall the proof when $\Omega$ is a bounded Lipschitz domain. For $h \in L^2(\Omega)$, consider the problem
    $$
    \begin{cases}
    \nabla \cdot F=h-\langle h\rangle_{\Omega} & \text {in } \Omega,
    \\ 
    F=0 & \text {on } \partial \Omega .
    \end{cases}
    $$
    Bogovskii's operator guarantees the existence of a solution $F \in (H_0^1(\Omega))^d$ satisfying
    \begin{equation*}
    \|F\|_{H^1(\Omega;\R^d)} \lesssim \|h-\langle h\rangle_{\Omega}\|_{L^2(\Omega)}
    \end{equation*}
    Then we have
    $$
    \|h-\langle h\rangle_{\Omega}\|_{L^2(\Omega)}^2
    =
    \int_{\Omega} (h-\langle h\rangle_{\Omega}) (\nabla \cdot F)=-\int_\Omega \nabla h \cdot F 
    \leq
    \|\nabla h\|_{H^{-1}(\Omega)}\|F\|_{H^1(\Omega)},
    $$
    which yields
    \begin{equation*}
        \|h-\langle h \rangle_{\Omega}\|_{L^{2}(\Omega)}
        \lesssim \|\nabla h\|_{H^{-1}(\Omega)}.
    \end{equation*}
    This is exactly \eqref{eq:from_AM} for $h=\langle f\rangle_V$.
\end{proof}

Building on Theorem \ref{th:poincare2}, we prove here Theorem \ref{th:poincare1}. Let us first show how we use the regularity assumption on the boundary of $\Omega$. Basically, it ensures that the two subsets $\Gamma_-$ and $\Gamma_+$ can be topologically separated, at least around one point of $\Gamma$.

\begin{lemma}\label{th:lemma_poincare}
    Let $\Omega$ be a bounded Lipschitz domain such that $\partial\Omega$ is $C^1$ at least at one point $x_0\in \partial\Omega$. Then, there exists a function $\varphi_0\in C_c^\infty(\overline{\D})$ satisfying $\varphi_0=0$ on $\Gamma_+ \cup \Gamma_0$, $v \cdot \nabla_x \varphi_0 \in L^2(\D)$, and
    \begin{equation}\label{eq:lemma_poincare}
        \frac{1}{|\D|}\int_\D v \cdot \nabla_x \varphi_0 \ dz =1.
    \end{equation}
\end{lemma}
\begin{proof}
    The proof is adapted from \cite[Theorem 1.3, Step 5]{Albritton2024VariationalEquation}. The $C^1$ regularity of the boundary at $x_0\in\partial\Omega$ ensures that the vector field $n_x$ is continuous at $x_0$. Thus, if $v_0\in{V}$ is such that $n_{x_0}\cdot v_0 <0$, then there exists $r>0$ such that $n_{x}\cdot v <0$ for $(x,v) \in (B_r(x_0) \cap \partial\Omega)\times (B_r(v_0)\cap {V})$. In other words, there exists $r>0$ such that $(B_r(x_0) \cap \partial\Omega)\times (B_r(v_0)\cap {V}) \subseteq \Gamma_-$. Therefore, we select a non-positive function $\varphi_0\in C_c^\infty(\overline{\D})$ with support in $(B_r(x_0) \cap \overline{\Omega})\times (B_r(v_0)\cap {V})$ and $\varphi_0(x_0,v_0)<0$. By the formula of integration by parts, we have that
    \begin{equation*}
        \frac{1}{|\D|}\int_\D v \cdot \nabla_x \varphi_0 = \frac{1}{|\D|}\int_\Gamma (n_x \cdot v) \varphi_0 >0,
    \end{equation*}
    and the integral on the right-hand side can be made equal to $1$ after multiplying $\varphi_0$ by a suitable constant.
\end{proof}

We can now address the Poincaré inequality for $\Hyptr(\D)$.

\begin{proof}[Proof of Theorem \ref{th:poincare1}]
    Let $f \in \Hyptr(\D)$. To prove \eqref{eq:poincare}, it is sufficient to show \eqref{eq:zz_mean_poincare} and combine it with \eqref{eq:poincareD}.
    Consider the test function $\varphi_0$ constructed in Lemma \ref{th:lemma_poincare}. We first use \eqref{eq:lemma_poincare} to split the average of $f$ as
    \begin{equation}
    \label{eq:zz_poincare_terms}
        \langle f \rangle_{\D}
        = \frac{1}{|\D|} \int_{\D} f (v \cdot \nabla_{x} \varphi_{0}) \ dz - \frac{1}{|\D|} \int_{\D}(f-\langle f \rangle _{\D}) (v \cdot \nabla_{x} \varphi_{0}) \ dz.
    \end{equation}
    For the first integral, by \eqref{eq:transA},
    \begin{equation}\label{eq:zz_ibp}
    \begin{aligned} 
        \left|\frac{1}{|\D|}\int_{\D} f (v \cdot \nabla_{x} \varphi_0) \  dz \right| 
        =&
        \left|-\frac{1}{|\D|}\langle v \cdot \nabla_xf, \varphi_0 \rangle + \frac{1}{|\D|} \int_{\Gamma}(n_x \cdot v) f_\Gamma \varphi_0 \ d\sigma\right| \\
        \lesssim &
        \|\varphi_0\|_{L^{2}_xH^1_v(\D)} \|v \cdot \nabla_xf \|_{L^{2}_x H^{-1}_v(\D)} \\
        & \hspace{3cm} + \|\varphi_0\|_{L^{2}(\Gamma_-, |n_x \cdot v|)} \|f_\Gamma\|_{L^{2}(\Gamma_-, |n_x \cdot v|)} \\
        \lesssim& \|v \cdot \nabla_xf \|_{L^{2}_x H^{-1}_v(\D)} + \|f_\Gamma\|_{L^{2}(\Gamma_-, |n_x \cdot v|)}
        ,
    \end{aligned}
    \end{equation}
    where we used that $\varphi_0$ vanishes on $\Gamma_+$. For the second term in \eqref{eq:zz_poincare_terms}, we use that $v \cdot \nabla_x \varphi_0 \in L^2(\D)$ to get
    \begin{align*}
        \left|\frac{1}{|\D|} \int_{\D} (f-\langle f \rangle_{\D}) (v \cdot \nabla_{x} \varphi_0) \ dz \right| 
        & \leq \frac{1}{|\D|} \|f-\langle f \rangle_{\D}\|_{L^{2}(\D)}\|v \cdot \nabla_{x} \varphi_0\|_{L^{2}(\D)} \\
        & \lesssim \|f-\langle f \rangle_{\D}\|_{L^{2}(\D)}.
    \end{align*}
    This concludes the proof of \eqref{eq:zz_mean_poincare}.
\end{proof}

\begin{remark}\label{rm:differences_poincare}
    As anticipated in Section \ref{sec:comm_poincare}, the proof of Theorem \ref{th:poincare1} is modelled on the proof of \cite[Theorem 1.3]{Albritton2024VariationalEquation} for the space $\HypO(\D)$ introduced in \eqref{eq:trace0_armstrong}. Because of the definition of $\HypO(\D)$ itself, in the proof of \cite[Theorem 1.3]{Albritton2024VariationalEquation}, the authors can work with smooth, compactly supported functions, and run the computations in \eqref{eq:zz_ibp} (cf. \cite[Equation (3-9)]{Albritton2024VariationalEquation}) using integration by parts%
    \footnote{ They claim to rely on $f$ being smooth also in \eqref{eq:zz_ibp_false} (cf. the second equation after \cite[Equation (3-5)]{Albritton2024VariationalEquation}), but the distributional definition of $v\cdot \nabla_xf$ for test functions in $C^1_c(\D)$ is enough to justify that step, and no smoothness is required.
    }%
    . On the contrary, here we do not rely on any density argument, and integration by parts against $\varphi_0$ is ensured by the definition of the function space $\Hyptr(\D)$ itself.
\end{remark}

\section{Renormalisation and comparison principle}
\label{sec:renorm}

\subsection{Renormalisation formula}

Renormalisation is the key tool for energy estimates. We adapt a result of Zhu \cite[Lemma 2.5]{Zhu2025RegularityDomains} and Avelin and Hou \cite[Lemma 5.3]{Avelin2026WeakDomains}.

\begin{lemma}[Renormalisation formula]\label{th:normalisation_kolm}
    Let $c\in\R$, and consider a renormalising function
    $\chi \in C^{0,1}(\R)$ such that $\chi'\in W^{1,\infty}(\R)$. 

    \begin{enumerate}[label=(\roman*)]
    \item \label{item:renorm1}
    Let $\Omega$ be a {$C^{1,\alpha_V}$} bounded domain, and let $f$ be a weak solution to \eqref{eq:kolm_both} in the interior of $\D$. 
    Then, for every test function $\phi\in C^1_c(\overline{\D})$
    it holds
    \begin{multline}
        \langle S , \chi'(f)\phi \rangle
        = \int_{\D} \left[{\nabla}_v\chi(f) \cdot {\nabla}_v\phi +  \chi''(f) |{\nabla}_vf|^2 \phi - \chi(f)(v\cdot \nabla_x \phi) + cf\chi'(f)\phi \right] dz \\
        + \int_{\Gamma} (n_x \cdot v) \chi(\tr_xf) \phi \, d\sigma.
    \label{eq:normalisation_kolm1}
    \end{multline}
    In particular, if $f\in\Hyptr(\D)$, then $f$ is a solution up to the boundary with trace $\tr_xf$.
    \item \label{item:renorm1_infty}
    Let $\Omega$ be a bounded Lipschitz domain, and let $f \in L^\infty(\D)$ be a weak solution to \eqref{eq:kolm_both} in the interior of $\D$. Then there exists $f_\Gamma\in L^\infty(\Gamma)$ such that $f$ is a solution up to the boundary with trace $f_\Gamma$. Moreover, formula \eqref{eq:normalisation_kolm1} holds and
    \begin{equation}\label{eq:trace_energy0}
        \|f_\Gamma\|_{L^\infty(\Gamma)} \leq \|f\|_{L^\infty(\D)}.
    \end{equation}
    \item \label{item:renorm2}
    Let $\Omega$ be a {$C^{1,\alpha_V}$} bounded domain (or $\torus$), and let $f$ be a weak solution to Problem \ref{prob:refl}{} (or \ref{prob:torus}{}). Then, for every test function $\phi\in \Crefl$ (or $\varphi \in C^1_c(\D)$)
    it holds
    \begin{equation}
        \langle S , \chi'(f)\phi \rangle
        = \int_{\D} \left[{\nabla}_v\chi(f) \cdot {\nabla}_v\phi +  \chi''(f) |{\nabla}_vf|^2 \phi - \chi(f)(v\cdot \nabla_x \phi) + cf\chi'(f)\phi \right] dz.
    \label{eq:normalisation_kolm2}
    \end{equation}
    \end{enumerate}
\end{lemma}

\begin{proof}[Proof of Lemma \ref{th:normalisation_kolm}]
    Let us prove  \ref{item:renorm1} first. Let $\{U_i\}_{i=0}^n$ be an open cover for $\overline{\Omega} \subseteq \R^d$ satisfying
    \begin{gather*}
        \overline{U_0} \subseteq \Omega
        \qquad \text{and} \qquad
        \partial\Omega \subseteq \bigcup_{i=1}^n U_i .
    \end{gather*}
    Equation \eqref{eq:normalisation_kolm1} is linear in $\phi$: we shall prove it for $\phi\in C^1_c(\overline{\D}\cap (U_i \times {V}))$ for $i=0,\dots,n$, and finally sum all contributions by partition of unity to get the claim for $\phi\in C^1_c(\overline{\D})$.

    \emph{Step 1.}
    Let us first work on 
    $$\D_0 \coloneqq \D \cap (U_0\times {V}) = U_0\times {V}.$$
    Consider a standard mollifier sequence $\{\rho_\varepsilon\}_{\varepsilon>0} \subseteq C^\infty_c(\R^d)$, with $\supp \rho_\varepsilon \subseteq B_\varepsilon(0)$, $\rho_\varepsilon (x) = \varepsilon^{-d}\rho_1\left(\frac{x}{\varepsilon} \right)$. Given a function $h\in L^2(\D_0)$, for fixed $v\in{V}$ consider the mollification $h_\varepsilon(\cdot,v) \coloneqq h(\cdot,v) * \rho_\varepsilon$, which is well defined on $\overline{\D_0}$ for $\varepsilon$ sufficiently small. Also, by standard definition, $\langle S_\varepsilon , \varphi \rangle \coloneqq \langle S, \varphi_\varepsilon  \rangle$ for $\varphi\in C^1_c(\D_0)$. In particular, we have 
    \begin{gather*}
        f_\varepsilon \to f,\  {\nabla}_vf_\varepsilon \to {\nabla}_vf, \ ({\nabla}_vf)_\varepsilon \to {\nabla}_vf \quad \text{in} \ L^2(D_0), \\
        S_\varepsilon \to S \quad \text{in} \ L^2_xH_v^{-1}(\D_0).
    \end{gather*}
    Now, for every $y\in U_0$ and $\varphi\in C^1_c(\D_0)$, we pick the test function $\varphi(y,v) \rho_\varepsilon(y-x)$ and integrate \eqref{eq:kolm_rad_weak_interior} in $x\in U_0$:
    \begin{align*}
        \langle S_\varepsilon , \varphi \rangle 
        =& \int_{\D_0} \left[({\nabla}_vf)_\varepsilon(y,v) \cdot {\nabla}_v\varphi(y,v) + c f_\varepsilon(y,v) \varphi(y,v) \right] dy\,dv\\
        &- \int_{U_0\times U_0 \times {V}} \varphi(y,v)f(x,v)(v \cdot \nabla_x)\rho_\varepsilon(y-x) \,dx\,dy\,dv.
    \end{align*}
    Integration by parts shows that
    \begin{align*}
        - \int_{\D_0} f_\varepsilon(y,v)(v\cdot &\nabla_y \varphi(y,v)) \,dy\,dv \\
        &= \int_{U_0\times U_0 \times {V}} \varphi(y,v)f(x,v)(v \cdot \nabla_y)\rho_\varepsilon(y-x) \,dx\,dy\,dv\\
        &= - \int_{U_0\times U_0 \times {V}} \varphi(y,v)f(x,v)(v \cdot \nabla_x)\rho_\varepsilon(y-x) \,dx\,dy\,dv.
    \end{align*}
    Therefore, there exist functions
    $R_{\varepsilon}\in L^2(\D_0)$ such that $R_{\varepsilon} \to 0$ in $L^2(\D_0)$ as $\varepsilon\to0$, and
    \begin{equation*}
        \langle S_\varepsilon , \varphi \rangle 
        = \int_{\D_0} \left[({\nabla}_vf_\varepsilon + R_{\varepsilon})\cdot {\nabla}_v\varphi - f_\varepsilon(v\cdot \nabla_x \varphi) + c f_\varepsilon \varphi \right] dz.
    \end{equation*}
    Now, for any $\phi\in C^1_c(\D_0)$, consider the test function $\varphi \coloneqq \chi'(f_\varepsilon)\phi$. Passing to the limit as $\varepsilon \to 0$, we find the renormalisation formula on $\D_0$:
    \begin{equation*}
        \langle S , \chi'(f)\phi \rangle 
        = \int_{\D_0} \left[{\nabla}_v\chi(f) \cdot {\nabla}_v\phi +  \chi''(f) |{\nabla}_vf|^2 \phi - \chi(f)(v\cdot \nabla_x \phi) + cf\chi'(f)\phi \right] dz.
    \end{equation*}

    \emph{Step 2.}
    Let us now investigate one of the boundary charts, say 
    $$\D_1 \coloneqq \D \cap (U_1\times {V})
    \quad \text{with} \quad
    \Gamma_1 \coloneqq \Gamma \cap (U_1 \times {V}).$$ Without loss of generality, suppose that the boundary $\partial\Omega \cap U_1$ is the graph of a {$C^{1,\alpha_V}$} function along the $x_d$-axis, and $\Omega \cap U_1$ is its epigraph. Then there exists $\lambda>0$ such that $B_\varepsilon(x+\lambda\varepsilon e_d) \subseteq \Omega$ for every $x\in \Omega\cap U_1$ and $\varepsilon>0$ sufficiently small. Given a function $h\in L^2(\D_1)$, for fixed $v\in{V}$ consider the translated mollification in $x$ given by $h_\varepsilon(x,v) \coloneqq (h(\cdot,v) * \rho_\varepsilon)(x+\lambda\varepsilon e_d)$, which is well defined on $\overline{\D_1}$ for $\varepsilon>0$ sufficiently small. Analogously as before, define $S_\varepsilon$. For every $\varphi\in C^1_c(\D_1 \cup \Gamma_1)$, we pick the test function $\varphi(y,v) \rho_\varepsilon(y-x+\lambda\varepsilon e_d)$ and, by the same steps as before, we show that there exist functions
    $P_{\varepsilon}\in L^2(\D_1)$ such that $P_{\varepsilon} \to 0$ in $L^2(\D_1)$ as $\varepsilon\to0$, and such that
    \begin{equation}\label{eq:weak_ep}
        \langle S_\varepsilon , \varphi \rangle 
        = \int_{\D_1} \left[({\nabla}_vf_\varepsilon + P_{\varepsilon})\cdot {\nabla}_v\varphi - f_\varepsilon(v\cdot \nabla_x \varphi) + c f_\varepsilon \varphi \right] dz
        + \int_{\Gamma_1} (n_x \cdot v) f_\varepsilon \varphi \,d\sigma.
    \end{equation}
    Notice that $\varphi(y,v) \rho_\varepsilon(y-x+\lambda\varepsilon e_n)$ vanishes at $\Gamma_1$ for $\varepsilon$ small enough, and therefore it is a suitable test function for the formulation \eqref{eq:kolm_rad_weak_interior} in the interior of $\D$. Again, for every $\phi\in C^1_c(\D_1 \cup \Gamma_1)$, we choose $\varphi \coloneqq \chi'(f_\varepsilon)\phi$ and get
    \begin{align}
        \langle S , \chi'(f)\phi \rangle
        =& \int_{\D_1} \left[{\nabla}_v\chi(f) \cdot {\nabla}_v\phi +  \chi''(f) |{\nabla}_vf|^2 \phi - \chi(f)(v\cdot \nabla_x \phi) + cf\chi'(f)\phi \right] \,dz \notag \\
        &+ \lim_{\varepsilon\to 0} \int_{\Gamma_1} (n_x \cdot v) \chi(f_\varepsilon) \phi \,d\sigma.
        \label{eq:limit_normalisation}
    \end{align}
    By Propositions \ref{th:trace_2} and \ref{th:trace_p}, it holds $f_\varepsilon \to \tr_xf$ strongly in $L^2(\Gamma_1, \omega_2)$. Therefore, by Cauchy-Schwarz, we obtain
    \begin{multline}\label{eq:renorm_conclusion}
        \left|\int_{\Gamma_1} (v \cdot n_x) (\chi(f_\varepsilon)-\chi(\tr_x f)) \phi \,d\sigma \right| \\
        \lesssim 
        \left(\int_{\Gamma_1} \omega_2(x,v) |f_\varepsilon-\tr_x f|^2 \,d\sigma \right)^{\frac{1}{2}}\left(\int_{\Gamma_1} \frac{|v \cdot n_x|^2}{\omega_2(x,v)}|\phi|^{2}\,d\sigma \right)^{\frac{1}{2}} 
        \longrightarrow 0 .
    \end{multline}
   This concludes the proof of \eqref{eq:normalisation_kolm1}.

    \emph{Step 3.}
    Let us now consider \ref{item:renorm1_infty}. Let $\Omega^\delta$ be a smooth approximation of $\Omega$ for $\delta \to 0$, such that $\Omega \subseteq \Omega^\delta$, and define the smooth vector field $$n^\delta(x) \coloneqq - \frac{\nabla_x\dist(x,\partial\Omega^\delta) }{|\nabla_x\dist(x,\partial\Omega^\delta)|}h_\delta(x), \qquad x\in\overline{\Omega}\subseteq \overline{\Omega^\delta},$$ where $h_\delta$ is a cut-off function in a suitably small neighbourhood of $\partial\Omega^\delta$. In particular, $n^\delta(x) \to n_x$ for a.e. $x\in \partial\Omega$ (see, e.g., \cite[Theorem 1.12]{Verchota1984LayerDomains}).
    
    The proof works as in Steps 1 and 2 up to equation \eqref{eq:limit_normalisation}.
    To conclude, we will prove that we can pass the limit in the boundary term. For $\varepsilon,\eta>0$, subtract the corresponding equations \eqref{eq:weak_ep} for $f_\varepsilon$ and $f_\eta$ respectively, and use the test functions $\varphi(x,v) \coloneqq (n^\delta(x) \cdot v) {\eta_R(v)}(f_\varepsilon-f_\eta)(x,v)$, {where $\eta_R\in C_c^\infty(B_{2R}(0))$, $R>1$, is a cutoff function satisfying $\eta_R=1$ on $B_R(0)$%
    \footnote{
    Notice that for $V=\Sd$ the cutoff can be omitted.
    }
    }. {For any fixed $\delta$ and $R$,} passing to the limit as $\varepsilon,\eta \to 0$ we obtain
    \begin{equation}\label{eq:critical_step}
        \lim_{\varepsilon,\eta\to 0} \int_{\Gamma_1} \left[(n_x \cdot v)^2 + (n^\delta(x)-n_x)\cdot v (n_x \cdot v)\right](f_\varepsilon-f_\eta)^2 {\eta_R}\,d\sigma =0.
    \end{equation}
    The term $(n^\delta(x)-n_x)\cdot v (n_x \cdot v)$ is uniformly bounded for $|v| \leq 2R$ and vanishes almost everywhere. Moreover, by the smoothness of $f_\varepsilon$ and Young's inequality, we have 
    \begin{equation}\label{eq:zz_weak*1}
        \|f_{\varepsilon}\|_{L^{\infty}(\Gamma)} 
        \leq\|f_{\varepsilon}\|_{L^{\infty}(\D)} \leq\|f\|_{L^{\infty}(\D)}.
    \end{equation}
    Then, the term $(f_\varepsilon-f_\eta)^2$ is also uniformly bounded. Thus we can take the limit for $\delta\to 0$ and obtain that $(n_x\cdot v)f_\varepsilon$ strongly converges to some function $F$ in $L^2_{\loc}(\Gamma_1)$. 
    Moreover, up to extracting a subsequence in $\varepsilon$, equation \eqref{eq:zz_weak*1} shows there exists a function $f_{\Gamma} \in L^{\infty}(\Gamma)$ such that
    \begin{equation}\label{eq:zz_weak*2}
        f_{\varepsilon} \xrightharpoonup{*} f_{\Gamma} \quad \text{in} \ L^{\infty}(\Gamma), 
        \qquad
        \|f_{\Gamma}\|_{L^{\infty}(\Gamma)}
        \leq \|f\|_{L^{\infty}(\D)} .
    \end{equation}Taking the limit for $\varepsilon\to 0$ in \eqref{eq:weak_ep}, we also get that 
    $$\int_{\Gamma_1} (n_x \cdot v) f_{\Gamma} \phi = \int_{\Gamma_1} F \phi \qquad \text{for every} \quad \phi\in C^1_c(\overline{\D}\cap (U_1 \times {V})),$$ 
    which yields $F=(n_x \cdot v) f_{\Gamma}$. This is enough to pass the limit in \eqref{eq:limit_normalisation}.

    \emph{Step 4.}
    Consider now \ref{item:renorm2}. For Problem \ref{prob:torus}{},  the proof works as in Step 1. For Problem \ref{prob:refl}{}, since in particular $f$ is a solution in the interior of $\D$, formula \eqref{eq:normalisation_kolm1} holds. In particular, the boundary term vanishes for $\phi\in \Crefl$ thanks to \eqref{eq:hyprefl2}, proving \eqref{eq:normalisation_kolm2}.    
\end{proof}

\begin{remark}
    The regularity assumption on $\Omega$ for items \ref{item:renorm1} and \ref{item:renorm2} in the renormalisation lemma is used to carry the limits of the boundary terms. If one assumes only $\Omega$ Lipschitz as in \ref{item:renorm1_infty}, some other strategy must be adopted. In particular, exchanging the limits $\delta \to 0$ and $\varepsilon, \eta \to 0$ in \eqref{eq:critical_step} requires some uniform boundedness. Notice that assuming $f$ solution up to the boundary for some trace $f_\Gamma$ (that is, having a trace already instead of obtaining a candidate from the weak-$*$ limit \eqref{eq:zz_weak*2}) does not help in carrying out the limit, as it does not ensure the required uniform boundedness in \eqref{eq:critical_step}.
\end{remark}

\begin{remark}\label{rm:regularity_1/4}
    Inequality \eqref{eq:renorm_conclusion} shows that, for $V=\Sd$, assuming that $\Omega$ is $C^{1,\alpha}$ for $\alpha>1/4$ would suffice. Indeed, by Proposition \ref{th:trace_p}, in this case the trace operator with weight $\omega_p$ is bounded for $p=(1/\alpha)-1<3$. The right-hand side of \eqref{eq:renorm_conclusion} then would be 
    \[
    \left(\int_{\Gamma_1} \omega_p(x,v) |f_\varepsilon-f_{\Gamma}|^2 \,d\sigma \right)^{\frac{1}{2}}\left(\int_{\Gamma_1} \frac{|v \cdot n_x|^2}{\omega_p(x,v)}|\phi|^{2}\,d\sigma \right)^{\frac{1}{2}} ,
    \]
    where the second term is indeed bounded for $p<3$.
\end{remark}

\subsection{Comparison principle}

The renormalisation formula yields a comparison principle: it holds for $c\geq 0$ for solutions up to the boundary, and for $c>0$ for solutions to Problem \ref{prob:refl}{} or \ref{prob:torus}{}. 
In this section, for $c \in\R$ denote
\begin{equation*}
    \Lo_c  \coloneqq - {\A_v} + v \cdot \nabla_x + c.
\end{equation*}
{In this section, whenever $\Omega$ is a bounded domain we always assume $C^{1,\alpha_V}$ regularity, so that the renormalisation formulas apply.}

\begin{proposition}[Comparison principle]\label{th:comparison}
    Let $f$ be a weak solution to \eqref{eq:kolm_both} in the interior of $\D$.
    \begin{enumerate}[label=(\roman*)]
        \item \label{item:bound_on_boundary} 
        If $f\leq 0$ in $\D$, then $\tr_xf \leq 0$ on $\Gamma$.
        \item \label{item:boundary_gives_bound_AG} For $c\geq0$, if $\Lo_c f \leq 0$ in $\D$ and $\tr_xf \leq0$ on $\Gamma_-$, then $f \leq 0$ in $\D$.
    \end{enumerate}
    Let $f$ be a solution to Problem \ref{prob:refl}{} or \ref{prob:torus}{}.
    \begin{enumerate}[label=(\roman*), resume]
        \item \label{item:boundary_gives_bound_BC} For $c>0$, if $\Lo_c f \leq 0$ in $\D$, then $f \leq 0$ in $\D$.
    \end{enumerate}
\end{proposition}

\begin{proof}
For \ref{item:bound_on_boundary}, consider
\begin{equation}\label{eq:approx_t+}
    \chi(t)= 
    \begin{cases}
        0 & \text{if} \ t \leq 0, \\ 
        t^2 & \text{if} \ 0<t<1, \\
        2t-1 & \text{if} \ t \geq 1.
    \end{cases}
\end{equation}
As $f\leq 0$, the renormalisation formula gives
\begin{equation*}
    0=\langle S , \chi'(f)\phi \rangle = \int_\Gamma (n_x \cdot v) \chi(\tr_xf) \phi \,d\sigma.
\end{equation*}
This holds for every $\phi\in C^1_c(\overline{\D})$, thus $\tr_xf\leq 0$.

For \ref{item:boundary_gives_bound_AG}, the proof is modelled on \cite[Theorem 3]{Avelin2026WeakDomains}. Consider the following real functions for $k>0$:
\begin{equation}\label{eq:zz_square_cut}
    \psi_{k}(t)= 
    \begin{cases}
        0 & \text{if} \ t \leq 0 \\ 
        t & \text{if} \ 0<t<k \\
        k & \text{if} \ t \geq k
    \end{cases}
    \qquad \text{and} \qquad
    \Psi_{k}(t)=\int_{-\infty}^{t} \psi_{k}(s) \ ds.
\end{equation}
The renormalisation formula with $\chi = \Psi_{k}$ and a non-negative radial test function $\phi=\phi{_R}(v)$ {for $R\geq 1$, with $\phi_R \in C^\infty_c(B_{2R}(0))$, $\phi_R=1$ on $B_R(0)$ and $|\nabla_v\phi_R| \lesssim R^{-1}$}, 
gives\footnote{
Again, $\phi=1$ is enough if $V=\Sd$.
}
\begin{multline}\label{eq:zz_last_points}
    \langle S , \psi_k(f) {\phi_R} \rangle
    = \int_\D \left[{\nabla_v\Psi_k(f)\cdot \nabla_v \phi_R} + \psi_k'(f) |{\nabla}_vf|^2 {\phi_R}+ c f\psi_k(f){\phi_R} \right] dz \\ + \int_\Gamma (n_x \cdot v) \Psi_k(\tr_xf) {\phi_R} \,d\sigma.
\end{multline}
Since $S\leq 0$, $\tr_xf \leq 0$ on $\Gamma_-$, and $c\geq 0$, {in the limit $R \to + \infty$} we get
\begin{equation}\label{eq:zz_last_points2}
    \int_{f^{-1}(0,k)} \left[ |{\nabla}_vf_+|^2 + cf_+^2 \right] dz
    + \int_{\Gamma_+} (n_x \cdot v) \Psi_k(\tr_xf) \,d\sigma\leq 0.
\end{equation}
Since the inequality holds for every $k>0$, it follows that $\tr_xf \leq 0$ on $\Gamma$ and ${\nabla}_vf_+=0$ in $\D$ (and also $f\leq 0$ already, if $c>0$). In particular, a standard Poincaré inequality in $v$, namely \eqref{eq:zz_stand_poincare}, yields $f_+=f_+(x)$. We apply again the renormalisation formula, this time with the renormalising function \eqref{eq:approx_t+} and a test function $\phi \geq 0$:
\begin{align*}
    0 &\geq \langle S, \chi'(f)\phi \rangle - \int_\D cf\chi'(f)\phi \,dz \\
    &= - \int_\D \chi(f)(v \cdot \nabla_x \phi) \,dz
    = - \int_\Omega \chi(f) \left(\int_{{V}}v \cdot \nabla_x \phi \,dv \right) dx. 
\end{align*}
Then, it suffices to choose $\phi \geq 0$ such that $\int_{{V}}v \cdot \nabla_x \phi < 0$ for a.e. $x\in\Omega$ to conclude $f_+=0$. For instance, take $R>0$ such that $\Omega \subseteq B_R(0)$, and consider the function $\phi(x,v) = (2R - x \cdot v){\phi_{1}(v)}$.

For \ref{item:boundary_gives_bound_BC}, the proof is the same up to equations \eqref{eq:zz_last_points} and \eqref{eq:zz_last_points2}, where the boundary term is omitted. As $c>0$, we deduce $f\leq 0$ directly.
\end{proof}

\begin{remark}
    Statements \ref{item:bound_on_boundary} and \ref{item:boundary_gives_bound_AG} hold, in particular, for solutions to Problem \ref{prob:trace}. In view of Remark \ref{rm:sol_boundary_probl}, statement \ref{item:boundary_gives_bound_BC} holds, in particular, for solutions to Problems \ref{prob:trace_Ro} and \ref{prob:trace_To}.
\end{remark}

\begin{proposition}[Weak maximum principle]\label{th:max_princ}
    Let $f$ be a weak solution to \eqref{eq:kolm_both} up to the boundary.
    \begin{enumerate}[label=(\roman*)]
        \item \label{eq:energy_A}
        If $c>0$, then
        \begin{equation*}
            \|\tr_xf\|_{L^\infty(\Gamma)} \leq \|f\|_{L^\infty(\D)}
            \lesssim_c \|S\|_{L^\infty(\D)}+\|\tr_xf\|_{L^\infty(\Gamma_-)}.
        \end{equation*}
        \item \label{eq:zz_second_max}
        If $c\geq 0$ and $S\leq 0$, then 
        \begin{equation*}
            \sup_{\Gamma_+} \tr_xf \leq \sup_{\Gamma_-} \ (\tr_xf)_{+} 
            \qquad \text{and} \qquad
            \sup_{\D} f \leq \sup_{\Gamma_-} \ (\tr_xf)_{+} .
        \end{equation*}
        In particular, if $c\geq 0$ and $S=0$, then
        \begin{equation*}
            \|\tr_xf\|_{L^\infty(\Gamma_+)}\leq \|\tr_xf\|_{L^\infty(\Gamma_-)} 
            \qquad \text{and} \qquad
            \|f\|_{L^\infty(\D)} = \|\tr_xf\|_{L^\infty(\Gamma_-)} .
        \end{equation*}
    \end{enumerate}

    Let $f$ be a weak solution to Problem \ref{prob:trace_Ro} or  \ref{prob:trace_To}.
    \begin{enumerate}[label=(\roman*), resume]
        \item \label{eq:energy_ART}
        If $c>0$, then
        \begin{equation*}
            \|\tr_xf\|_{L^\infty(\Gamma)} \leq \|f\|_{L^\infty(\D)}
            \lesssim_c \|S\|_{L^\infty(\D)}.
        \end{equation*}
    \end{enumerate}
    
    Let $f$ be a solution to Problem \ref{prob:refl}{} or \ref{prob:torus}{}.
    \begin{enumerate}[label=(\roman*), resume]
        \item \label{eq:energy_BC}
        If $c>0$, then
        \begin{equation*}
            \|f\|_{L^\infty(\D)}
            \lesssim_c \|S\|_{L^\infty(\D)}.
        \end{equation*}
    \end{enumerate}
\end{proposition}
\begin{proof}
    Without loss of generality, suppose that all the terms on the right-hand side of the claims are finite. 
    To show \ref{eq:energy_A}, apply Proposition \ref{th:comparison}\ref{item:boundary_gives_bound_AG} and compare $f$ to the functions
    $$h^\pm\coloneqq \pm\left(\frac{1}{c}\|S\|_{L^\infty(\D)}+\|\tr_xf\|_{L^\infty(\Gamma_-)}\right),$$
    which solve, respectively,
    \begin{equation*}
        \Lo_c h^\pm = \pm\left(\|S\|_{L^\infty(\D)}+c\|\tr_xf\|_{L^\infty(\Gamma_-)} \right).
    \end{equation*}
    Similarly, claims \ref{eq:energy_BC} and \ref{eq:energy_ART} follow from Proposition \ref{th:comparison}\ref{item:boundary_gives_bound_BC} -- recall Remark \ref{rm:sol_boundary_probl}. Moreover, for \ref{eq:energy_A} and \ref{eq:energy_ART}, the first inequality follows from \eqref{eq:trace_energy0}; same for one side of the equality in \ref{eq:zz_second_max}.
    To show the other inequality in \ref{eq:zz_second_max}, consider $M \coloneqq \sup_{\Gamma_-} (\tr_xf)_{+} \geq 0$, which solves $\Lo_c M = cM$. Therefore $\Lo_c f = S\leq 0 \leq cM = \Lo_c M$ in $\D$, and clearly $\tr_xf \leq M$ on $\Gamma_-$. Conclude with the comparison principle.
\end{proof}

\begin{corollary}[Uniqueness]\label{th:uniqueness}
    Let $f_1$ and $f_2$ be both solutions to either
    \begin{enumerate}[label=(\roman*)]
        \item Problem \ref{prob:trace_G} for $c\geq0$, or 
        \item Problem \ref{prob:trace_Ro}, \ref{prob:trace_To}, \ref{prob:refl}{} or \ref{prob:torus}{} for $c>0$.
    \end{enumerate}
    Suppose $\Lo_c f_1=\Lo_c f_2 \in L^2_xH_v^{-1}(\D)$. Then $f_1=f_2$.
\end{corollary}

\section{Well-posedness and one-sided trace estimate}
\label{sec:well-posedness_proof}

\subsection{The inflow boundary problem} 
\label{sec:well_inflow}
We now prove Theorem \ref{th:wellpos_inflow_initial}. First, we prove the existence of a solution for $c>0$: at this stage, $\Omega$ is required to be Lipschitz only.

\begin{lemma}\label{th:existence1}
    Let $c>0$, $\Omega$ a bounded Lipschitz domain, $G\in L^2(\Gamma_-, |n_x \cdot v|)$. Then, there exists a weak solution $f$ to Problem \ktr{$(\Bo_{G})$}.
\end{lemma}

We give two proofs of this fact: via the Lions-Lax-Milgram theorem, which we state below, and via the vanishing viscosity method, which recovers a solution in the limit $\varepsilon \to 0$ of the elliptic problem associated to  $\varepsilon \Delta_x + {\Delta}_v$. However, the latter applies only for $V=\Sd$.

\begin{theorem}[J.-L. Lions \cite{Lions1957SurCylindriques}]\label{th:lions}
    Let $(F, \|\cdot\|_F)$ be a Hilbert space, and let $(H,\| \cdot \|_H)$ be a pre-Hilbert space continuously embedded in $F$. Suppose $E \colon F \times H \to \R$ is a bilinear form such that:
    \begin{enumerate}[label={\bf[H\arabic*]},ref={\bf[H\arabic*]}, leftmargin=4em]
        \item \label{item:H1} for every $h\in H$, the linear form $f \mapsto E[f,h]$ is continuous on $F$;
        \item \label{item:H2} there exists $\alpha>0$ such that $|E[h,h]| \geq \alpha \|h\|^2_H$ for every $h\in H$.
    \end{enumerate}
    Then, for every linear form $L\in H^*$, there exists a vector $f\in F$ such that $E[f,-]=L$.
\end{theorem}

\begin{proof}[Proof of Lemma \ref{th:existence1}]
    Consider
    \begin{equation*}
        F \coloneqq L^2_{x}H^1_v(\D)\times L^2(\Gamma_+, |n_x\cdot v|)
    \end{equation*}
    with norm
    \begin{equation}\label{eq:norm_F}
        \|(f,g)\|^2_F \coloneqq \int_{\D} |f|^2 \,dz + \int_{\D}|{\nabla}_vf|^2 \,dz+ \int_{\Gamma_+}(n_x\cdot v)_+|g|^2 \,d\sigma,
    \end{equation}
    which is a Hilbert space, and consider
    \begin{equation*}
        H\coloneqq C^1_c(\overline{\D})
    \end{equation*}
     with norm
    \begin{equation}\label{eq:norm_H}
        \|h\|^2_H \coloneqq \int_{\D} |h|^2 \,dz + \int_{\D}|{\nabla}_vh|^2 \,dz+ \int_{\Gamma}|n_x\cdot v||\Tr_xh|^2 \,d\sigma,
    \end{equation}
    which is a pre-Hilbert space. Clearly, the embedding $h \mapsto (h, \Tr_xh)$ of $H$ in $F$ is continuous.
    Consider the bilinear form $E$ given by
    \begin{equation*}
        E[(f,g), h] \coloneqq \int_{\D} \left[{\nabla}_vf \cdot {\nabla}_vh - f(v\cdot \nabla_x h) + c f h \right] dz
        + \int_{\Gamma_+} (n_x \cdot v)_+ g \Tr_xh \,d\sigma,
    \end{equation*}
    and the continuous functional 
    $$L(h)\coloneqq \langle S, h \rangle + \int_{\Gamma_-} |n_x \cdot v| G \Tr_xh \,d\sigma.$$
    For every $h\in H$ and $(f,g)\in F$, it holds
    \begin{equation*}
        |E[(f,g),h]| \lesssim_c \left(\|h\|_{H}+\|v \cdot \nabla_x h\|_{L^2(\D)} \right) \|(f,g)\|_F,
    \end{equation*}
    which is \ref{item:H1}. Given $h\in H$, integration by parts yields
    \begin{equation*}
        \int_{\D} h(v \cdot \nabla_xh) \,dz= \frac{1}{2} \int_{\Gamma}(n_x \cdot v) |\Tr_x h|^2 \,d\sigma.
    \end{equation*}
    Therefore,
    \begin{equation}\label{eq:coercivity_E}
        E[h,h] = \int_{\D} \left[|{\nabla}_vh|^2 + c|h|^2 \right] dz+ \frac{1}{2} \int_{\Gamma}
        |n_x\cdot v||\Tr_xh|^2 \,d\sigma
        \gtrsim \|h\|_H^2,
    \end{equation}
    which gives \ref{item:H2}. Thus, there exists a pair $(f,g) \in F$ such that $E[(f,g),h]=L(h)$ for every $h\in H$. In particular, writing
    \begin{equation}\label{eq:def_f_gamma}
        f_\Gamma \coloneqq 
        \begin{cases}
            g & \text{on} \ \Gamma_+,\\
            G & \text{on} \ \Gamma_-,
        \end{cases}
    \end{equation}
    we find that $(f,f_\Gamma)$ satisfies \eqref{eq:kolm_rad_weak}. In particular, $f$ solves Problem \ktr{}($\Bo_G$). 
\end{proof}

\begin{remark}
\label{rm:comm_avelin}
    In the proof above, we applied the Lions-Lax-Milgram theorem including the boundary term in the bilinear form, obtaining a solution with trace in $L^2(\Gamma,|v\cdot n_x|)$ directly. In \cite{Avelin2026WeakDomains} this point was left open; there, optimal trace integrability was instead recovered from Zhu’s vanishing-viscosity construction -- see \cite[Remark 3.6]{Avelin2026WeakDomains}.
\end{remark}

\begin{proof}[Alternative proof of Lemma \ref{th:existence1} for $V =\Sd$]
    For $\varepsilon \in(0,1)$, let $f_{\varepsilon}$ be the weak solution to the elliptic problem
    \begin{equation}\label{eq:elliptic}
        \left\{
        \begin{array}{lll}
        - \varepsilon \Delta_{x} f_{\varepsilon} - {\A_v} f_{\varepsilon} +v \cdot \nabla_{x} f_{\varepsilon}
        + c f_\varepsilon = S & \text {in} & \D, \\[2mm]
        \varepsilon n_{x} \cdot \nabla_{x} f_{\varepsilon}+\left(n_{x} \cdot v\right)_{-} f_{\varepsilon}=(n_x \cdot v )_-G &\text {on} & \Gamma,
        \end{array}
        \right.
    \end{equation}
    whose weak formulation for any $\varphi \in H^{1}(\D)$ is
    \begin{equation}\label{eq:elliptic_weak}
        \langle S , \varphi \rangle + \int_{\Gamma_-} |n_x \cdot v |G \Tr_x \varphi \, d\sigma
        = B[f_\varepsilon,\varphi],
    \end{equation}
    where the left-hand side is continuous in $H^1(\D)$, and
    \begin{align*}
    B[f_\varepsilon,\varphi]  
        \coloneqq& \int_{\D} \left[\varepsilon {\nabla}_xf_\varepsilon \cdot {\nabla}_x \varphi + {\nabla}_vf_\varepsilon \cdot {\nabla}_v\varphi - f_\varepsilon(v\cdot \nabla_x \varphi) + c f_\varepsilon \varphi \right] dz
        + \int_{\Gamma} (n_x \cdot v)_+ f_\varepsilon \varphi \,d\sigma\\
        =& \int_{\D} \left[\varepsilon {\nabla}_xf_\varepsilon \cdot {\nabla}_x \varphi + {\nabla}_vf_\varepsilon \cdot {\nabla}_v\varphi + c f_\varepsilon \varphi \right] dz \\
        &- \frac{1}{2} \int_\D \big[f_\varepsilon(v\cdot \nabla_x \varphi) - \varphi(v\cdot \nabla_x f_\varepsilon) \big] dz
        + \frac{1}{2}\int_{\Gamma} |n_x \cdot v| f_\varepsilon \varphi \,d\sigma.
    \end{align*}
    As $c>0$, we can apply the classical Lax-Milgram theorem to deduce that the problem \eqref{eq:elliptic} has a unique solution. We wish to find the solution to \eqref{eq:kolm_both} in the limit $\varepsilon \to 0$. We can derive a standard energy bound by choosing the solution $f_{\varepsilon}$ itself as a test function, and applying the Cauchy-Schwarz inequality:
    \begin{equation*}
        \int_{\Gamma}|n_{x} \cdot v| f_{\varepsilon}^{2} \,d\sigma+ \int_{\D} \left[f_{\varepsilon}^{2} + |{\nabla}_{v} f_{\varepsilon}|^{2}+\varepsilon|\nabla_{x} f_{\varepsilon}|^{2}\right] dz
        \lesssim_c \|S\|^2_{L^2_xH_v^{-1}(\D)} + \|G\|_{L^2(\Gamma_-, |n_x \cdot v|)}^2.
    \end{equation*}
    Hence, up to subsequence, there exist $f \in L_x^{2}H^1_v(\D)$ and $f_{\Gamma} \in L^{2}(\Gamma,|n_{x} \cdot v| )$ that satisfy 
    \begin{align*}
        f_{\varepsilon} \rightharpoonup f
        \quad &\text{in} \ L^2_xH^1_v(\D), \\
        \varepsilon \nabla_{x} f_{\varepsilon} \to 0 
        \quad &\text{in} \ L^{2}(\D), \\
        f_{\varepsilon}|_\Gamma \rightharpoonup g 
        \quad &\text{in} \ L^{2}(\Gamma, |n_{x} \cdot v|),
    \end{align*}
    as $\varepsilon \rightarrow 0$. Taking the limit in \eqref{eq:elliptic_weak}, we deduce that the weak formulation \eqref{eq:kolm_rad_weak} holds for the limiting pair $(f,f_{\Gamma})$ for $f_\Gamma$ defined as in \eqref{eq:def_f_gamma}. In particular, $f_{\Gamma}=G$ on $\Gamma_-$.
\end{proof}

\begin{remark}
    The vanishing viscosity method does not apply to the model $V=\Rg$ because, in that case, the term $\int_\D f_\varepsilon(v \cdot \nabla_x \varphi) \, dz$ would not be defined for $\varphi\in H^1(\D)$.
\end{remark}

We can now conclude the proof of Theorem \ref{th:wellpos_inflow_initial}. {To obtain energy estimates, we use the renormalisation formula: this is the only reason to assume $C^{1,\alpha_V}$ regularity for the bounded domain $\Omega$.}

\begin{proof}[Proof of Theorem \ref{th:wellpos_inflow_initial}]
    Let first $c>0$. Consider the solution $f$ with trace $f_\Gamma$ provided by Lemma \ref{th:existence1}. Its uniqueness follows by Corollary \ref{th:uniqueness}.
    Recall that, since we are also assuming that $\Omega$ is a $C^{1,\alpha_V}$ domain, in fact $f_\Gamma= \tr_xf$. Apply the renormalisation formula with $\chi(\iota)=\iota^{2}$ as renormalisation function%
    \footnote{
    More precisely, similarly to what is done with the functions $\phi_k$ defined in \eqref{eq:zz_square_cut}, take $\chi_k$ an appropriate adjustment of $\chi(\iota)=\iota^{2}$ which satisfies the assumptions of Lemma \ref{th:normalisation_kolm}, and then take the limit for $k \to \infty$. The same trick holds for similar functions later on.
    }
    and {$\phi=\phi_R$ from the proof of Proposition \ref{th:comparison}} as test function: using the Cauchy-Schwarz inequality, {in the limit $R \to \infty$} we find
    \begin{equation}\label{eq:zz_energy_w/out}
        \int_{\D} f^{2} \,dz+\int_{\Gamma}|n_{x} \cdot v| f_\Gamma^{2} \,d\sigma+\int_{\D}|{\nabla}_{v} f|^{2} \,dz\lesssim_c \|S\|^2_{L^2_xH_v^{-1}(\D)}+\int_{\Gamma_-} |n_{x} \cdot v|G^{2} \,d\sigma.
    \end{equation}
    From the weak formulation \eqref{eq:kolm_rad_weak} itself, recalling inequality \eqref{eq:explain_bounded} we recover that
    \begin{equation*}
        \|v\cdot \nabla_xf\|_{L^2_xH^{-1}_v(\D)}
        \lesssim_c \|f\|_{L^{2}(\D)} + \|{\nabla}_{v} f\|_{L^{2}(\D)} + \|S\|_{L^2_xH_v^{-1}(\D)},
    \end{equation*} 
    which concludes the proof of \eqref{eq:energy1}.
    The uniqueness follows from Corollary \ref{th:uniqueness}.

    Consider now the case $c=0$. We obtain the solution as the limit for $c\to 0$ of the sequence of solutions for $c>0$. Let $f_c\in \Hyptr(\D)$ be the unique solution to the problem for $c\in(0,1)$, and denote $g_c \coloneqq (f_c)_\Gamma$. In the derivation of the energy estimate \eqref{eq:energy1}, we can keep track of the coefficient $c$ and get
    \begin{gather}
        \label{eq:zz_with_c}
        c\|f_c\|^2_{L^{2}(\D)}+\|{\nabla}_{v} f_c\|^2_{L^{2}(\D)}+\|g_c\|^2_{L^{2}(\Gamma, |n_x\cdot v|)}
        \lesssim \frac{1}{c}\|S\|^2_{L^2_xH_v^{-1}(\D)} + \|G\|^2_{L^{2}(\Gamma_-, |n_x\cdot v|)}, \\
        \label{eq:hyp_norm}
        \|v \cdot \nabla_xf_c\|_{L^2_xH^{-1}_v(\D)}
        \lesssim c \|f_c\|_{L^{2}(\D)} + \|{\nabla}_{v} f_c\|_{L^{2}(\D)} + \|S\|_{L^2_xH_v^{-1}(\D)}. 
    \end{gather}
    Notice how \eqref{eq:zz_with_c} alone is useless in the limit $c\to0$. However, thanks to the Poincaré inequality \eqref{eq:poincare}, we can obtain an energy estimate independent of $c\in(0,1)$, which allows us to pass to the limit. 
    Indeed, \eqref{eq:zz_with_c} and \eqref{eq:hyp_norm} together give
    \begin{equation*}
        \|v \cdot \nabla_xf_c\|_{L^2_xH^{-1}_v(\D)} \lesssim \|{\nabla}_{v} f_c\|_{L^{2}(\D)} + \|S\|_{L^2_xH_v^{-1}(\D)} + \|G\|_{L^{2}(\Gamma_-, |n_x\cdot v|)},
    \end{equation*}
    and the Poincaré inequality for $f_c\in \Hyptr(\D)$ entails
    \begin{equation}
    \label{eq:zz_without_c}
        \|f_c\|_{L^{2}(\D)} + \|v \cdot \nabla_xf_c\|_{L^2_xH^{-1}_v(\D)} \lesssim \|{\nabla}_{v} f_c\|_{L^{2}(\D)} + \|S\|_{L^2_xH_v^{-1}(\D)} + \|G\|_{L^{2}(\Gamma_-, |n_x\cdot v|)},
    \end{equation}
    Moreover, for $\chi(\iota)=\iota^{2}$ and {$\phi=\phi_R$ for $R \to \infty$}, the renormalisation formula gives
    \begin{equation}
    \label{eq:zz_inter_hyp}
    \begin{aligned}
        \int_{\Gamma_-} |n_{x} \cdot v| G^2 \,d\sigma + 2 \langle S, f_c \rangle 
        &= 2\int_{\D}|{\nabla}_{v} f_c|^{2} \,dz+ 2c\int_{\D} f_c^{2} \,dz + \int_{\Gamma_+}|n_{x} \cdot v| g_c^{2} \,d\sigma\\
        &\geq \int_{\D}|{\nabla}_{v} f_c|^{2} \,dz + \int_{\Gamma_+}|n_{x} \cdot v| g_c^{2} \,d\sigma.
    \end{aligned}
    \end{equation}
    A combination of \eqref{eq:zz_without_c}, \eqref{eq:zz_inter_hyp} and the Cauchy-Schwarz inequality yields
    \begin{equation}\label{eq:hypo_estimate}
        \int_{\D} f_c^{2} \,dz + \int_{\D}|{\nabla}_{v} f_c|^{2} \,dz + \int_{\Gamma}|n_{x} \cdot v| g_c^{2} \,d\sigma
        \lesssim \|S\|^2_{L^2_xH_v^{-1}(\D)} + \int_{\Gamma_-} |n_{x} \cdot v| G^2 \,d\sigma,
    \end{equation} 
    where the implicit constant is independent of $c$.
    
    Now let $c_i\in(0,1)$ for $i\in\{1,2\}$. The difference of the respective solutions $f_{c_i}$ satisfies
    \begin{align*}
        \int_\D (c_2-c_1)f_{c_2} \varphi \,dz
        =& \int_{\D} \left[{\nabla}_v(f_{c_1}-f_{c_2}) \cdot {\nabla}_v\varphi - (f_{c_1}-f_{c_2})v\cdot \nabla_x \varphi + c_1(f_{c_1}-f_{c_2}) \varphi \right] dz \\
        &+ \int_{\Gamma} (n_x \cdot v) (g_{c_1}-g_{c_2}) \varphi \,d\sigma.
    \end{align*}
    Thus, by estimate \eqref{eq:hypo_estimate},
    \begin{multline*}
        \int_{\D} (f_{c_1}-f_{c_2})^{2} \,dz + \int_{\D}|{\nabla}_{v} (f_{c_1}-f_{c_2})|^{2} \,dz + \int_{\Gamma}|n_{x} \cdot v| (g_{c_1}-g_{c_2})^{2} \,d\sigma \\
        \lesssim |c_1-c_2| \int_{\D}f_{c_2}^{2} \,dz
        \lesssim |c_1-c_2| \left(\|S\|^2_{L^2_xH_v^{-1}(\D)} + \int_{\Gamma_-} |n_{x} \cdot v| G^2 \,d\sigma\right).
    \end{multline*} 
    This yields the existence of a limiting pair $(f,f_\Gamma)$ as $c \rightarrow 0$, which satisfies \eqref{eq:kolm_rad_weak} and $f_{\Gamma}=G$ on $\Gamma_-$, for $c=0$. In particular, it satisfies  \eqref{eq:hyp_norm} and thus, by the same steps, estimate \eqref{eq:hypo_estimate}.
\end{proof}

\subsection{A one-sided trace estimate}
\label{sec:trace_theorem}

As a consequence of Theorem \ref{th:wellpos_inflow_initial} and Remark \ref{rm:sol_boundary_probl}, we can characterise $\Hyptr(\D)$ as the set of solutions to \eqref{eq:kolm_both} up to the boundary. Indeed, for $f\in\Hyptr(\D)$, consider 
\begin{equation}\label{eq:S_for_f}
    S \coloneqq v \cdot \nabla_xf-{\A_v}f+f \in L^2_xH_v^{-1}(\D).
\end{equation}
Trivially, $f$ solves Problem \ktr{}($\Bo_{f_\Gamma}$) with $c=1$. This fact is enough to prove the one-sided trace estimate for $\Hyptr(\D)$. 

\begin{proof}[Proof of Theorem \ref{th:strong_trace_0}]
    For $f\in\Hyptr(\D)$, consider $S$ as in \eqref{eq:S_for_f}. Then, \eqref{eq:trace_partial1} follows from estimate \eqref{eq:energy1} and the bounds \begin{equation*}
        \|{\A_v}f\|_{L^2_xH_v^{-1}(\D)}\leq \|{\nabla}_vf\|_{L^2(\D)}
        \quad \text{and}
    \quad \|f\|_{L^2_xH_v^{-1}(\D)} \leq \|f\|_{L^2(\D)}.
    \end{equation*}
    For \eqref{eq:trace_partial2}, simply consider ${f}(x,v) \coloneqq f(x,-v)$, with trace ${f}_\Gamma (x,v) \coloneqq f_\Gamma(x,-v)$, and apply \eqref{eq:trace_partial1}: indeed, ${f}\in \Hyptr(\D)$, with
    \begin{equation*}
        \|{f} \|_{\Hyp(\D)} = \left\|{f} \right\|_{\Hyp(\D)}
        \quad \text{and} \quad
        \|{f}_\Gamma \|_{L^2(\Gamma_\pm,|n_x \cdot v|)} = \left\|f_\Gamma \right\|_{L^2(\Gamma_\mp,|n_x \cdot v|)}. \qedhere
    \end{equation*}
\end{proof}

\begin{remark}
Instead of \eqref{eq:kolm_both}, one could study the equation
\begin{equation}
\label{eq:kolm_rad2}
    (c - v \cdot \nabla_x) f = {\A_v}f+  S,
    \qquad 
    (x,v)\in \D,
\end{equation}
where we have changed the sign in front of the transport term. The well-posedness proof is identical, with the only difference that the roles of $\Gamma_-$ and $\Gamma_+$ are swapped. In particular, the inequality \eqref{eq:trace_partial2} is related to the well-posedness of the inflow boundary problem for \eqref{eq:kolm_rad2}, in the same way \eqref{eq:trace_partial1} was obtained from the inflow problem for \eqref{eq:kolm_both}.
\end{remark}

\subsection{The specular reflection and periodic problems}
\label{sec:refl_torus}
We finally turn to well-posedness for Problems \ref{prob:refl}{} and \ref{prob:torus}{}.

\begin{proof}[Proof of Theorem \ref{th:wellpos_refl_initial}]
    Fix $c>0$. Using the same argument as in the proof of the inflow problem, we find a solution $f$ by applying the Lions-Lax-Milgram theorem to the weak formulation \eqref{eq:kolm_rad_weak_refl} or \eqref{eq:kolm_rad_weak_torus}, with $F = L^2_{x}H^1_v(\D)$, and $H=\Crefl$ or $C^1_c(\D)$ respectively, with $L^2_{x}H^1_v(\D)$-norm. Then, taking $\chi(\iota)=\iota^{2}$ and $\phi=\phi_R$ (as in the proof of Theorem \ref{th:wellpos_inflow_initial}) in the renormalisation formula and using the Cauchy-Schwarz inequality, we obtain the energy estimate 
    \begin{equation}\label{eq:zz_with_c2}
        \|f\|_{L^{2}(\D)}+\|{\nabla}_{v} f\|_{L^{2}(\D)}      \lesssim_c \|S\|_{L^2_xH_v^{-1}(\D)}.
    \end{equation}
    Again, from the equation itself we get $v\cdot\nabla_xf \in L^2_xH^{-1}_v(\D)$ and the estimate \eqref{eq:energy6.1}.
    Consider now the case $c=0$.  If $f_c\in L^2_xH^1_v(\D)$ is the unique solution to the problem for $c >0$, the condition $\langle S \rangle_{\D}=0$ implies $\langle f_c \rangle_{\D}=0$: simply choose $\varphi=\phi_R$ (from the proof of Proposition \ref{th:comparison}) as test function, and take the limit $R \to \infty$. We consider the limit for $c\to 0$ as done in the proof of Theorem \ref{th:wellpos_inflow_initial}, and clearly the limit function $f$ also satisfies $\langle f \rangle_{\D}=0$. Notice that we apply the Poincaré inequality \eqref{eq:poincareD} for functions in $\Hyp(\D)$ with zero average.
\end{proof}

\begin{proof}[Proof of Proposition \ref{th:extend_to_boundary_BC}]
    Consider the solution to Problem \ref{prob:refl}{}. Since $f$ is a solution in the interior of $\D$ and $f\in L^\infty(\D)$, Lemma \ref{th:normalisation_kolm}\ref{item:renorm1_infty} implies that $f$ is a weak solution to \eqref{eq:kolm_both} up to the boundary for some trace $f_\Gamma\in L^\infty(\Gamma)$. Also, by uniqueness of the trace, $f_\Gamma = \tr_x f$, which satisfies $\tr_xf= \Bo_\Ro \tr_xf$ on $\Gamma_-$. Therefore $f$ solves Problem \ref{prob:trace_Ro}. The same argument applies to Problem \ref{prob:torus}{}, as it is easy to see that the trace $f_{\partial\D}\in L^\infty(\partial\D)$ satisfies $f_{\partial\D}= \Bo_\To f_{\partial\D}$ on $\Gamma_-$.
\end{proof}

As a consequence of Proposition \ref{th:extend_to_boundary_BC}, we obtain a partial well-posedness result for Problems \ref{prob:trace_Ro} and \ref{prob:trace_To} when the source is essentially bounded. 
This result was already established by Zhu \cite{Zhu2025RegularityDomains} for the time-dependent equation with operator $\Bo_\Ro$, with a direct, constructive proof: we adapt and report it in Appendix \ref{app:auxiliary} for the sake of completeness.

\begin{proposition}\label{prop:refl_torus_trace}
    Let $c>0$ and $\Omega$ a bounded $C^{1,\alpha_V}$ domain. Consider Problem \ref{prob:trace_Ro} or \ref{prob:trace_To}. Assume also $S\in L^\infty(\D)$. Then, there exists a unique solution $f$, which satisfies
    \begin{gather}
        \label{eq:energy5.1}
        \|f\|_{L^{2}(\D)} + \|{\nabla}_{v} f\|_{L^{2}(\D)}
        \lesssim_c \|S\|_{L^2_xH_v^{-1}(\D)} \\
        \label{eq:energy5.2}
        \|f_{\Gamma}\|_{L^\infty(\Gamma)} \leq \|f\|_{L^\infty(\D)}
        \lesssim_c \|S\|_{L^\infty(\D)}.
    \end{gather}
\end{proposition}
\begin{proof}
        Existence is given by Proposition \ref{th:extend_to_boundary_BC} and Proposition \ref{th:max_princ}\ref{eq:energy_BC}, and uniqueness by Corollary \ref{th:uniqueness}. Estimate \eqref{eq:energy5.1} is proved by a now standard application of the renormalisation formula, whereas \eqref{eq:energy5.2} was already proved in Proposition \ref{th:max_princ}\ref{eq:energy_ART}. 
\end{proof}

\section{The bounded-velocity model}
\label{sec:bounded}

We now consider equation \eqref{eq:kolm_both} when the velocity variable is confined to a bounded domain $\V\subseteq\R^d$.
The corresponding Dirichlet problem was studied by Avelin and Hou \cite{Avelin2026WeakDomains}. We retain their formulation and complement their well-posedness result with energy estimates, which we obtain thanks to the renormalisation formula under the $C^{1,1}$ regularity assumption on $\Omega$.

We keep the same notation for $\D=\Omega\times\V$ and $\Gamma=\partial \Omega \times \V$, and analogously for $\Gamma_\pm$ and $\Gamma_0$. This time, the velocity boundary $\Omega \times \partial\V$ is also involved. Denote by $\Tr_v$ the standard $v$-trace operator for $L^2_xH^1_v(\D)$, and define
\begin{equation}
    \HOv(\D) \coloneqq \{\varphi\in H^1(\D) \mid
    \Tr_v \varphi=0 \ \text{on } \Omega\times\partial\V\}.
\end{equation}
The space $\Hyptr(\D)$ is defined as in Definition \ref{def:hypA}: however, in this setting,
 \begin{equation}
    L^2_xH^1_{0,v}(\D)
    \coloneqq
    L^2(\Omega;H^1_0(\V)),
    \qquad
    L^2_xH^{-1}_v(\D) \coloneqq \big(L^2_xH^1_{0,v}(\D)\big)' ,
\end{equation}
and \eqref{eq:transA} is tested only for $\varphi \in \HOv(\D)$.

\begin{definition}[Bounded-velocity Dirichlet problem] \label{def:Dirichlet_bounded}
    Let $\Omega, \V$ be bounded Lipschitz domains, $S\in L_x^2H_v^{-1}(\D)$ and $c \in \R$. Consider an inflow boundary datum $G_1\in L^2\bigl(\Gamma_-,|n_x \cdot v|\bigr)$ and a lifted velocity boundary datum $G_2\in H^1(\D)$.
    A {weak solution to the Dirichlet problem}
    \begin{equation}\label{eq:dirichlet}
        \begin{cases}
            cf + v \cdot \nabla_x f = \A_vf+  S
            &\text{in} \ \D
            \\
            f = G_1
            &\text{on} \ \Gamma_-
            \\
            f = \Tr_v G_2
            &\text{on} \ \Omega \times \partial\V
        \end{cases}
    \end{equation}
    is a function $f\in L^2_xH^1_v(\D)$ for which there exists a (unique) $f_{\Gamma} \in L^2(\Gamma, |n_x \cdot v|)$, called \emph{trace} of $f$, such that 
    $$ \Tr_vf= \Tr_vG_2 \ \text{on} \ \Omega \times \partial\V,
    \qquad \text{and} \qquad
    f_\Gamma=G_1 \ \text{on} \ \Gamma_-,$$
    and for every $\varphi\in \HOv({\D})$ the pair $(f,f_\Gamma)$ satisfies
    \begin{equation}\label{eq:kolm_bounded_weak}
        \langle S , \varphi \rangle
        = \int_{\D} \left[{\nabla}_vf \cdot {\nabla}_v\varphi - f(v\cdot \nabla_x \varphi) + c f \varphi\right] dz
        + \int_{\Gamma} (n_x \cdot v) f_{\Gamma} \varphi \ d\sigma.
    \end{equation}
\end{definition}

Notice that, for $\varphi \in \HOv(\D)$, the Poincar\'e inequality on $\V$ gives
\begin{equation}
    \|\varphi\|_{L^2(\D)}
    \lesssim
    \|\nabla_v\varphi\|_{L^2(\D)},
    \label{eq:v-poincare}
\end{equation}
providing coercivity already for $c=0$. This makes the bounded-velocity case simpler than the velocity models considered above.

\begin{theorem}[Well-posedness for the bounded-velocity Dirichlet problem]
\label{thm:bounded-velocity-wellposedness}
Let $\Omega$ be a bounded $C^{1,1}$ domain, and let $\V\subseteq\R^d$ be a bounded Lipschitz domain.
Let $c\geq0$, $S\in L^2_xH^{-1}_v(\D)$, $G_1\in L^2\bigl(\Gamma_-,|n_x \cdot v|\bigr)$, and $G_2\in H^1(\D)$. Then there exists a unique solution $f$ to the Dirichlet problem \eqref{eq:dirichlet}. It satisfies $f\in\Hyptr(\D)$ and 
\begin{equation}\label{eq:energy_bounded_vel}
    \|f\|_{\Hyptr(\D)} 
    \lesssim_c
    \|S\|_{L^2_xH^{-1}_v(\D)}
    +
    \|G_1\|_{L^2(\Gamma_-,|n_x \cdot v|)}
    +
    \|G_2\|_{H^1(\D)}
\end{equation}
\end{theorem}

\begin{proof}
Suppose first $G_2=0$.
We apply the Lions-Lax-Milgram theorem as in the proof of Lemma \ref{th:existence1}. Consider
\[
    F \coloneqq 
    L^2_xH^1_{0,v}(\D)
    \times
    L^2\bigl(\Gamma_+,(n_x \cdot v) \bigr).
\]
with norm \eqref{eq:norm_F}, and $H\coloneqq \HOv(\D)$ with norm \eqref{eq:norm_H}. Proceed as in the proof of Lemma \ref{th:existence1}, considering the same bilinear form $E[(f,g),h]$ and the same functional $L(h)$. Notice that, this time, thanks to the $v$-Poincaré inequality \eqref{eq:v-poincare}, the coercivity estimate \eqref{eq:coercivity_E} for $E$ also holds for $c=0$. The Lions-Lax-Milgram theorem therefore yields a pair $(f,g)\in L^2_xH^1_{0,v}(\D) \times L^2(\Gamma_+,|n_x \cdot v|)$ such that $E[(f,g),h]=L(h)$ for every $\varphi\in\HOv(\D)$.
Defining
\[
    f_\Gamma
    \coloneqq 
    \begin{cases}
        g
        &\text{on }\Gamma_+,\\
        G_1
        &\text{on }\Gamma_-,
    \end{cases}
\]
we get that $f$ solves the Dirichlet problem \eqref{eq:kolm_bounded} with inflow datum $G_1$ and zero velocity-boundary condition.

From the equation itself, we get
\begin{align}
    \|v\cdot\nabla_xf\|_{L^2_xH^{-1}_v(\D)}
    &\lesssim
    \|S\|_{L^2_xH^{-1}_v(\D)}
    +
    \|\nabla_vf\|_{L^2(\D)}
    +
    c\|f\|_{L^2(\D)}.
    \label{eq:transport-estimate}
\end{align}
The renormalisation formula \cite[Lemma 5.3]{Avelin2026WeakDomains} -- which is the equivalent of Lemma \ref{th:normalisation_kolm} for the bounded-velocity model -- with renormalising function $\chi(t)=t^2$ and test function $\phi=1$ gives
\begin{equation*}
    \int_{\D} |\nabla_vf|^2\, dz + c\int_{\D} f^2\, dz +
    \frac12 \int_{\Gamma_+} (n_x \cdot v)|f_\Gamma|^2 \,d\sigma
    =
    \langle S,f\rangle_
        {L^2_xH^{-1}_v,L^2_xH^1_{v}}
    +
    \frac12
    \int_{\Gamma_-}
        |n_x \cdot v|| G_1|^2
        \,d\sigma.
\end{equation*}
There is no contribution from $\Omega\times\partial\V$ since $f\in L^2_xH^1_{0,v}(\D)$.
By \eqref{eq:v-poincare} and Young's inequality, we get
\begin{align}
    \|\nabla_vf\|_{L^2}^2
    +
    c\|f\|_{L^2}^2
    +
    \|f_\Gamma\|_{L^2(\Gamma_+,|n_x \cdot v|)}^2
    \lesssim
    \| S\|_{L^2_xH^{-1}_v}^2
        +
        \| G_1\|_
            {L^2(\Gamma_-,|n_x \cdot v|)}^2.
    \label{eq:h-coercive-energy}
\end{align}
Using Poincar\'e inequality \eqref{eq:v-poincare} once more, and combining with \eqref{eq:transport-estimate}, we get 
\begin{align}
    \|f\|_{\Hyptr(\D)} 
    \lesssim_c
    \| S\|_{L^2_xH^{-1}_v}
    +
    \| G_1\|_{L^2(\Gamma_-,|n_x \cdot v|)}.
    \label{eq:h-energy}
\end{align}
This concludes the case $G_2=0$.

For arbitrary $G_2 \in H^1(\D)$, define
\begin{align*}
    \tilde S
    &\coloneqq 
    S+\Delta_vG_2-v\cdot\nabla_xG_2-cG_2 \in
    L^2_xH^{-1}_v(\D), \\
    \tilde G_1
    &\coloneqq 
    G_1-\Tr_xG_2 \in L^2(\Gamma_-, |n_x \cdot v|).
\end{align*}
Here $\Delta_vG_2$ is understood distributionally. Let $\tilde f$ be the solution of \eqref{eq:kolm_bounded} with source $\tilde S$ and inflow datum $\tilde G_1$. Finally, $f\coloneqq \tilde f+G_2$ solves \eqref{eq:kolm_bounded} with trace $f_\Gamma \coloneqq \tilde f_\Gamma+\Tr_xG_2$. Moreover, $f=\Tr_vG_2$ on $\Omega\times\partial\V$, and $f_\Gamma=G_1$ on $\Gamma_-$, and estimate \eqref{eq:energy_bounded_vel} follows from \eqref{eq:h-energy}. Uniqueness is a consequence of the energy estimate, but it can also be proved independently, similarly to Proposition \ref{th:comparison} and Corollary \ref{th:uniqueness}: see \cite[Theorem 3]{Avelin2026WeakDomains}.
\end{proof}

\appendix

\section{The auxiliary problems}
\label{app:auxiliary}
Let $\Omega$ be a bounded Lipschitz domain. To prove Proposition \ref{prop:refl_torus_trace}, we adapt the techniques used in \cite{Zhu2025RegularityDomains} for the time-dependent equation with specular reflection. The (unique) solutions to Problem \ref{prob:trace_Ro} and \ref{prob:trace_To} can be found as the limit of an appropriate sequence of solutions to the inflow boundary value problem. However, in contrast with the inflow problem, here the integral boundary term
$$\int_{\Gamma} (n_x \cdot v)f_{\Gamma}^2 \,d\sigma$$
vanishes if $f_{\Gamma}$ satisfies either the specular reflection condition $f_{\Gamma}=\Bo_\Ro f_{\Gamma}$ or the periodic condition $f_{\Gamma}= \Bo_\To f_{\Gamma}$ on $\Gamma_-$. Indeed, for  $\Bo \in \{\Bo_\Ro, \Bo_\To\}$
\begin{equation}\label{eq:zz_boundary_no}
    \int_{\Gamma} (n_{x} \cdot v)_{-} f_{\Gamma}^{2} \,d\sigma
    = \int_{\Gamma} (n_{x} \cdot v)_{-}(\Bo f_{\Gamma})^{2} \,d\sigma
    = \int_{\Gamma} (n_{x} \cdot v)_{+} f_{\Gamma}^{2} \,d\sigma,
\end{equation}
and the information on the $L^2$-norm of the trace is lost in the limit process. Therefore, we use the auxiliary boundary operators $a\Bo_\Ro$ and $a\Bo_\To$ for $a\in(0,1)$, for which the boundary flux does not vanish, combined with a limit argument holding for $S\in L^\infty(\D)$. 

\begin{lemma}\label{th:lemma_a}
    Let $c>0$ and $a \in[0,1)$, and consider Problem \ktr{}($a\Bo_\Ro$) or \ktr{}($a\Bo_\To$). Then, there exists a unique solution $f$, which satisfies
    \begin{equation}\label{eq:energy3}
        \|f\|_{L^{2}(\D)} + \|{\nabla}_{v} f\|_{L^{2}(\D)} + (1-a)\|f_{\Gamma}\|_{L^{2}(\Gamma, |n_x \cdot v|)} 
        \lesssim_c \|S\|_{L^2_xH_v^{-1}(\D)}.
    \end{equation}
    If also $S\in L^\infty(\D)$, then
    \begin{equation}\label{eq:energy4}
        \|f_{\Gamma}\|_{L^\infty(\Gamma)} \leq \|f\|_{L^\infty(\D)}
        \lesssim_c \|S\|_{L^\infty(\D)}.
    \end{equation}
\end{lemma}
\begin{proof}
    Let $\Bo$ be either $\Bo_\Ro$ or $\Bo_\To$. Through an iterative application of Theorem \ref{th:wellpos_inflow_initial}, for every $n\in\N$ consider the weak solution $f_n$ with trace $g_n \coloneqq (f_n)_\Gamma$ satisfying the inflow condition
    \begin{equation*}
    \begin{cases}
        g_{n}=a \Bo g_{n-1} & \text{for} \ n \geq 1 \\
        g_0=0
    \end{cases}
        \qquad \text{on} \ \Gamma_-.
    \end{equation*}
    As shown in \eqref{eq:zz_boundary_no},
    \begin{equation}\label{eq:zz_boundary}
        \int_{\Gamma} (n_{x} \cdot v)_{-} g_{n}^{2} \,d\sigma
        = \int_{\Gamma} (n_{x} \cdot v)_{-}(a \Bo g_{n-1})^{2} \,d\sigma
        = a^{2} \int_{\Gamma} (n_{x} \cdot v)_{+} g_{n-1}^{2} \,d\sigma.
    \end{equation}
    Taking $\chi(\iota)=\iota^{2}$ and $\phi=\phi_R$ (from the proof of Proposition \ref{th:comparison}) in the renormalisation formula, using the Cauchy-Schwarz inequality, in the limit $R\to\infty$ we obtain 
    \begin{gather}
        \int_{\D} f_{n}^{2} \,dz + \int_{\D}|{\nabla}_{v} f_{n}|^{2} \,dz
        \lesssim_c \|S\|^2_{L^2_xH_v^{-1}(\D)} + \int_{\Gamma}(n_{x} \cdot v)_{-} g_{n}^{2} \,d\sigma
        \label{eq:zz_partial1}\\
        \int_{\Gamma}(n_{x} \cdot v)_{+} g_{n}^{2} \,d\sigma
        \leq \|S\|^2_{L^2_xH_v^{-1}(\D)} + \int_{\Gamma}(n_{x} \cdot v)_{-} g_{n}^{2} \,d\sigma
        \label{eq:zz_partial2}
    \end{gather}
    Applying \eqref{eq:zz_boundary} and \eqref{eq:zz_partial2} iteratively leads to
    \begin{equation}\label{eq:zz_boundaryPos}
        \int_{\Gamma}(n_{x} \cdot v)_{+} g_{n}^{2} \,d\sigma
        \leq \sum_{i=0}^{n} a^{2 i} \|S\|^2_{L^2_xH_v^{-1}(\D)}
        \leq \frac{1}{1-a^{2}} \|S\|^2_{L^2_xH_v^{-1}(\D)}.
    \end{equation}
    By linearity, the function $f_{n}-f_{n-1}$ solves an inflow problem with $S=0$ and trace $g_{n}-g_{n-1}$ satisfying
    \begin{equation*}
        g_{n}-g_{n-1}=a \Bo(g_{n-1}-g_{n-2}) \quad \text{on} \ \Gamma_-.
    \end{equation*}
    Similarly, we get
    \begin{multline*}
        \int_{\D}(f_{n}-f_{n-1})^{2} \,dz+ \int_{\D}|{\nabla}_{v}(f_{n}-f_{n-1})|^{2} \,dz+ \int_{\Gamma}|n_{x} \cdot v|(g_{n}-g_{n-1})^{2} \,d\sigma \\
        \lesssim_c \int_{\Gamma}(n_{x} \cdot v)_{-}(g_{n}-g_{n-1})^{2} \,d\sigma 
        \leq a^{2 (n-1)} \int_{\Gamma}(n_{x} \cdot v)_{-}(g_{1}-g_{0})^{2} \,d\sigma.
    \end{multline*}
    Sending $n \to \infty$, we derive a pair of limiting functions $(f,f_{\Gamma})$ that solve \eqref{eq:kolm_rad_weak} with $f_{\Gamma}=a \Bo f_{\Gamma}$ on $\Gamma_-$, and that satisfy \eqref{eq:zz_partial1}. By specular reflection and \eqref{eq:zz_boundaryPos},
    \begin{equation*}
        \int_{\Gamma}|n_{x} \cdot v| f_{\Gamma}^{2} \,d\sigma
        = (1+a^2)\int_{\Gamma}(n_{x} \cdot v)_{+} f_{\Gamma}^{2} \,d\sigma
        \leq \frac{1+a^2}{1-a^{2}} \|S\|^2_{L^2_xH_v^{-1}(\D)}.
    \end{equation*}
    Combined with \eqref{eq:zz_partial1}, this proves \eqref{eq:energy3}.
    Uniqueness follows from the energy estimate.

    To show \eqref{eq:energy4}, denote $M \coloneqq \|S\|_{L^\infty(\D)}$ and consider the equation solved by
    $$\bar{f} \coloneqq f-\tfrac{1}{c}M,$$
    whose weak formulation is
    \begin{equation*}
        \int_{\D} (S -M) \varphi \,dz
        = \int_{\D} \left[{\nabla}_v\bar{f} \cdot {\nabla}_v\varphi - \bar{f}(v\cdot \nabla_x \varphi) + c \bar{f} \varphi \right] dz
        + \int_{\Gamma} (n_x \cdot v) \bar{f_{\Gamma}} \varphi \,d\sigma.
    \end{equation*}
    Applying the renormalisation formula with $\chi(\iota)=\iota_{+}^{2}$ and $\phi=\phi_R$ (from the proof of Proposition \ref{th:comparison}), in the limit $R\to\infty$ we get 
    \begin{equation*}
        0 \geq \int_{\D} 2 (S-M) \bar{f}_+ \,dz
         = \int_{\D} \left[ \chi''(\bar{f}) |{\nabla}_v\bar{f}|^2 + 2c\bar{f}\bar{f}_+ \right] dz + \int_{\Gamma} (n_x \cdot v) (\bar{f_{\Gamma}})^2_+ \,d\sigma
         \geq \int_{\D} 2c\bar{f}_+^2 \,dz,
    \end{equation*}
    where we used that
    \begin{align*}
        \int_{\Gamma}(n_{x} \cdot v)(f_{\Gamma}-\tfrac{1}{c}M)_{+}^{2} \,d\sigma
        &= \int_{\Gamma_+}(n_{x} \cdot v)\left[(f_{\Gamma}-\tfrac{1}{c}M)_{+}^{2}-(a f_{\Gamma}-\tfrac{1}{c}M)_{+}^{2} \right] \,d\sigma\\
        &\geq \int_{\Gamma_+}(n_{x} \cdot v)\left[(f_{\Gamma}-\tfrac{1}{c}M)_{+}^{2}-a^{2}(f_{\Gamma}-\tfrac{1}{c}M)_{+}^{2} \right] \,d\sigma\geq 0.
    \end{align*}
    In particular, it follows that $(f-\tfrac{1}{c}M)_+=\bar{f}_+=0$ in $\D$.
    Similarly, starting from the equation solved by $-f-\tfrac{1}{c}M$, one can show that $(-f-\tfrac{1}{c}M)_+=0$ in $\D$. Together with \eqref{eq:trace_energy0}, this yields \eqref{eq:energy4}.
\end{proof}

As anticipated, the above argument lacks information on the $L^2(\Gamma,|n_x \cdot v|)$ norm of the trace as $a\to 1$. Nevertheless, we can still find a candidate solution $f$ from \eqref{eq:energy3} in the limit $a\to 1$, whereas the trace $f_{\Gamma}$ is recovered using the extra assumption $S\in L^\infty(\D)$.

\begin{proof}[Proof of Proposition \ref{prop:refl_torus_trace}]
     Let $\Bo$ be either $\Bo_\Ro$ or $\Bo_\To$. For $i=1,2$, take $a_{i} \in[0,1)$, and let $f_{a_{i}}$ be the solution to Problem \ktr{}($a_i\Bo$) provided by Lemma \ref{th:lemma_a} -- more explicitly, the trace $g_{a_i}\coloneqq (f_{a_i})_\Gamma$ satisfies $g_{a_{i}}=a_{i} \Bo g_{a_{i}}$. In particular, $f_{a_{i}}$ satisfies \eqref{eq:energy5.1} and \eqref{eq:energy5.2}. Consider the equation solved by $f_{a_1}-f_{a_2}$: applying the renormalisation formula with $\chi(\iota)=\iota^{2}$ and $\varphi=1$, we find
     \begin{equation}\label{eq:zz_limit_a}
         2c \int_{\D}(f_{a_{1}}-f_{a_{2}})^{2} \,dz + \int_{\Gamma}(n_{x} \cdot v)(g_{a_{1}}-g_{a_{2}})^{2} \,d\sigma + 2 \int_{\D}|{\nabla}_{v}(f_{a_{1}}-f_{a_{2}})|^{2} \,dz = 0.
     \end{equation}
    The boundary term can be written as
    \begin{align*}
        \int_{\Gamma}&(n_{x} \cdot v)(g_{a_{1}}-g_{a_{2}})^{2} \,d\sigma\\
        &= \int_{\Gamma_+}(n_{x} \cdot v)\left[(g_{a_{1}}-g_{a_{2}})^{2}-(a_{1} g_{a_{1}}-a_{2} g_{a_{2}})^{2}\right] d\sigma \\
        &= \int_{\Gamma_+}(n_{x} \cdot v)\left[(1-a_{1}^{2})(g_{a_{1}}-g_{a_{2}})^{2} - 2 a_{1}(a_{1}-a_{2}) g_{a_{1}}(g_{a_{1}}-g_{a_{2}})-(a_{1}-a_{2})^{2} g_{a_{2}}^{2}\right] d\sigma.
    \end{align*}
    Owing to the estimate \eqref{eq:energy5.2} for $g_{a_{1}}, g_{a_{2}}$, the above boundary term tends to zero as $a_{1}, a_{2} \rightarrow 1$. Together with \eqref{eq:zz_limit_a}, this gives the existence of a limit function $f$ of $f_{a}$ as $a \rightarrow 1$ which satisfies \eqref{eq:energy5.1} and the second inequality in \eqref{eq:energy5.2}. At the same time, the estimate \eqref{eq:energy5.2} for $g_a$ yields the existence of a function $f_{\Gamma} \in L^{\infty}(\Gamma)$ such that $g_a \xrightharpoonup{*} f_\Gamma$ in $L^{\infty}(\Gamma)$, up to subsequence in $a$.
    Combining the two limits, we get that $(f,f_\Gamma)$ satisfies \eqref{eq:kolm_rad_weak} with $f_\Gamma=\Bo f_\Gamma$ on $\Gamma_-$, and \eqref{eq:energy5.2} follows from \eqref{eq:trace_energy0}.
\end{proof}

Note that, in particular, we have obtained Proposition \ref{th:max_princ}\ref{eq:energy_ART} constructively.

\bibliographystyle{abbrv}
\bibliography{references_kolm}

\end{document}